\documentclass[reqno]{amsart}
\usepackage{newcent}       
\usepackage{helvet}         
\usepackage{courier}        
\usepackage{amsfonts}
\usepackage{amsfonts,amssymb,amsmath}
\usepackage[latin1]{inputenc}
\usepackage{color}
\usepackage{graphicx}

\usepackage[active]{srcltx}
\newtheorem{theorem}{Theorem}[section]
\newtheorem{lemma}[theorem]{Lemma}
\newtheorem{proposition}[theorem]{Proposition}

\newtheorem{remark}[theorem]{Remark}
\newtheorem{definition}{Definition}
\newcommand{\mc}[1]{{\mathcal #1}}
\newcommand{\mf}[1]{{\mathfrak #1}}

\newcommand{\bb}[1]{{\mathbb #1}}

\newcommand{\eps}{\varepsilon}

\newcommand{\<}{\langle}
\renewcommand{\>}{\rangle}
\newcommand{\p}{\partial}
\newcommand{\pfrac}[2]{\genfrac{}{}{}{1}{#1}{#2}}

\author{Tertuliano Franco}
\address{IMPA\\
Estrada Dona Castorina, 110\\
Horto, Rio de Janeiro\\
Brasil}
\email{tertu@impa.br}

\author{Patr\'{\i}cia Gon\c{c}alves}
\address{CMAT, Centro de Matem\'atica da Universidade do Minho, Campus de Gualtar, 4710-057 Braga, Portugal}
\email{patg@math.uminho.pt}

\author{Adriana Neumann}
\address{IMPA\\
Estrada Dona Castorina, 110\\
Horto, Rio de Janeiro\\
Brasil}
\email{aneumann@impa.br}

\title[Exclusion process with slow bonds]{Hydrodynamical behavior of symmetric exclusion with slow bonds}


\date{}

\begin{document}

\begin{abstract}
We consider the exclusion process in the one-dimensional discrete torus with $N$ points,
where all the bonds have conductance one, except a finite number of slow bonds, with conductance $N^{-\beta}$, with $\beta\in[0,\infty)$. We prove that the time evolution of the empirical density of particles, in the diffusive scaling,
has a distinct behavior according to the range of the parameter $\beta$. If $\beta\in [0,1)$, the hydrodynamic limit is given by the usual heat equation.
If $\beta=1$, it is given by a parabolic equation involving an operator $\frac{d}{dx}\frac{d}{dW}$, where $W$ is the Lebesgue measure on the torus plus the sum of the Dirac measure supported on each macroscopic point related to the slow bond. If $\beta\in(1,\infty)$, it is given by
 the heat equation with Neumann's boundary conditions,
 meaning no passage through the slow bonds in the continuum.
\end{abstract}
\maketitle
\section{Introduction}

An important subject in statistical physics is the characterization of the hydrodynamical behavior of interacting particle systems in random or inhomogeneous media.
 One relevant and puzzling problem is to consider particle systems with slow bonds and to analyze the macroscopic effect on the hydrodynamic profiles, depending on the {\em{strength}} at these bonds. The problem we address in this paper is the complete characterization of the hydrodynamic limit scenario for the exclusion process with a finite number of slow bonds. Depending on the strength at the slow bonds, one observes a change of behavior that goes from smooth profiles to the development of discontinuities.

 We begin by giving a brief and far from complete review about some results on the subject, all of them related to the exclusion process. In
\cite{f}, by taking suitable random conductances $\{c_k:k\geq{1}\}$, such that $\{c_k^{-1}:k\geq{1}\}$ satisfy a Law of Large Numbers, it was proved that the
randomness of the medium does not survive in the macroscopic time evolution of the density of particles. In \cite{fjl}, the authors consider
conductances driven by an $\alpha$-stable subordinator $W$, and in this case, the randomness survives in the continuum, by replacing in the
hydrodynamical equation the usual Laplacian by a generalized operator $\pfrac{d}{dx}\pfrac{d}{dW}$, which results in the weak heat equation. In
the same line of such quenched result, \cite{fl} shows the analogous behavior, but for a general strictly increasing function $W$. All the
previous works are restricted to the one-dimensional setting, and strongly based on convergence results for diffusions or random walks in
 one-dimensional inhomogeneous media, see \cite{s}. In \cite{v}, there is a generalization of \cite{fl} for a suitable
$d$-dimensional setting, in some sense decomposable into $d$ one-dimensional cases. General sufficient conditions
for the hydrodynamical limit of exclusion process in inhomogeneous medium were established in \cite{j}.
All the above works have in common the association of the exponential clock with the \emph{bonds},
having the Bernoulli product measure as invariant measure, and being close, in some sense,
to the symmetric simple exclusion process.

In \cite{t}, the totally asymmetric simple exclusion process is considered to have a single bond with smaller
clock parameter. Such ``slow bond", not only slows down the passage of particles across it, but it also has a macroscopical impact since it disturbs the hydrodynamic profile.
Somewhat intermediate between the symmetric and asymmetric case, in \cite{b} is considered a single asymmetric bond
in the exclusion process. This unique asymmetric bond gives rise to a flux in the torus and also influences the macroscopic evolution of the density
of particles. In the symmetric case, \cite{fnv} obtained a $d$-dimensional result for a model in which the slow bonds are close
to a smooth surface.

As a consequence of the above results, one can observe the recurrent phenomena about the distinct characteristics of slow bonds in symmetric
and asymmetric settings. In the asymmetric case, e.g. \cite{t} and \cite{b}, the slow bond parameter does not need to be rescaled in order to have a macroscopic influence. Nevertheless, in the symmetric case, from \cite{fjl}, \cite{fl} and \cite{fnv}
we see that the slow bond must have parameter of order $N^{-1}$ in order to have macroscopical impact.

In this paper, we make precise this last statement for the following model. Consider the state space
of configurations with at most one particle per site in the discrete torus.
To each bond is associated an exponential clock. When this clock rings, the
occupancies of the sites connected by the bond are exchanged. All the bonds have clock parameter equal
to $1$, except $k$ finite bonds, chosen in such a way that these bonds correspond to $k$ fixed macroscopic points $b_1,\ldots,b_k$.
The conductances in these slow bonds are given by $N^{-\beta}$, with $\beta\in[0,+\infty)$ and the scale here is diffusive in all bonds.

If $\beta=1$, the time evolution of the density of particles $\rho(t,\cdot)$ is described by the partial differential equation
\begin{equation*}
\left\{
\begin{array}{l}
{\displaystyle \partial_t \rho \; =\; \pfrac{d}{dx}\pfrac{d}{dW} \rho } \\
{\displaystyle \rho(0,\cdot) \;=\; \gamma(\cdot)}
\end{array},
\right.
\end{equation*}
where the operator $\pfrac{d}{dx}\pfrac{d}{dW}$ is defined in Subsection \ref{weak operator} and $W$ is the Lebesgue measure on the torus plus the sum of the Dirac measure in each of the $\{b_i:i=1,...,k\}$.
This result is a particular case of both the results in \cite{fl} and \cite{fnv}. For the sake of completeness, we present here a simpler proof
of it. It is relevant to mention the interpretation of such partial differential equation as a weak version of
\begin{equation*}
\left\{
\begin{array}{l}
{\displaystyle \partial_t \rho \; =\; \partial_u^2 \rho } \\
 \partial_u\rho_t(1)\;=\;\partial_u\rho_t(0)\;=\;\rho_t(1)-\rho_t(0)\\
{\displaystyle \rho(0,\cdot) \;=\; \gamma(\cdot)}
\end{array}
\right.
\end{equation*}
where $0$ and $1$ mean the left and right side of a macroscopic point $b_i$ related to a slow bond.
This equation says that $\rho$ is discontinuous at each macroscopic point $\{b_i:i=1,...,k\}$ with passage of mass at such point
 and is governed by the {\em{Fick's Law: the rate of passage of mass is proportional to the gradient concentration}}.
Such interpretation comes from the natural domain $C_W$ of the operator $\pfrac{d}{dx}\pfrac{d}{dW}$,
defined in Subsection \ref{limit points critical beta}. It is easy to verify that all the functions in the domain
$C_W$ satisfy the above boundary condition, for more details see \cite{fl} and \cite{fnl}.

If $\beta\in[0,1)$, the conductances in these slow bonds do not converge to zero sufficiently fast in order to appear in the hydrodynamical limit. As a consequence, there is no macroscopical influence of the slow bonds in the continuum and we obtain
the hydrodynamical equation as the usual heat equation. The proof of last result is based on the {\em{Replacement Lemma}},
and the range parameter of $\beta$ is sharp in the sense that, it only works for $\beta\in[0,1)$.

As $\beta$ increases, the conductance at the slow bonds decreases and the passage of particles through these bonds becomes more difficult. In fact, for $\beta\in (1,+\infty)$, the clock parameters go to zero faster than at the critical value $\beta=1$
 and each slow bond gives rise to a barrier in the continuum. Macroscopically this phenomena gives rise to the usual heat equation with Neumann's boundary conditions at each macroscopic point $\{b_i:i=1,...,k\}$, which means here that the spatial derivative of $\rho$ at each $\{b_i:i=1,...,k\}$ equals to zero and, physically, this represents an {\em{isolated boundary}}.
Moreover, the uniqueness of weak solutions of such equation
says explicitly that the macroscopic evolution of the density of particles is independent for each
interval $[b_i,b_{i+1}]$, however the passage of particles in the discrete torus through the slow bonds is still possible.
The proof of this result is also based on the {\em{Replacement Lemma}} and requires sharp energy estimates.

Since the regime $\beta=1$ was already known from previous works,
the main contribution of this article is the complete characterization of the
three distinct behaviors for the time evolution of the empirical density of particles,
exhibiting a behavior change depending on the parameter of the conductance at the slow bonds.
From our knowledge, no
similar phenomena were exploited for the hydrodynamic limit of interacting particle systems.
Moreover, for the regime $\beta\in(1,\infty)$ the density evolves according to the heat equation with
Neumann's boundary conditions, which has a meaningful physical interpretation. This is the other great novelty developed in this paper. So far, partial differential equations with Dirichlet's boundary conditions could be approached by e.g. studying interacting particle systems in contact with reservoirs. Here, by considering partial differential equations with Neumann's boundary conditions, we give a step towards extending the set of treatable partial differential equations by the hydrodynamic limit theory. Besides all the mentioned achievements, we also prove that the regime
$\beta=1$ is critical, since the other two regimes have positive Lebesgue measure on the line.

In order to achieve our goal, the main difficulties appear in the characterization of limit points for each regime of $\beta$.
We overcome this difficulty by developing a suitable \emph{Replacement Lemma}, which allow us to replace product of site occupancies by functions of the empirical measure in the continuum limit.
Furthermore, that lemma is also crucial for characterizing the behavior near
the slow bonds.

Our result can also be extended to non-degenerate exclusion type models as introduced in \cite{glt}. In such models, particles interact with hard core exclusion and the rate of
exchange between two consecutive sites is influenced by the number of particles in the vicinity of the exchanging sites. The jump rate is strictly positive, so that all the configurations are erdogic, in the sense that a move to an unoccupied site can always occur. It was shown in \cite{glt} that the hydrodynamical equation for such models is given by a non-linear partial equation.
Having established the \emph{Replacement Lemma}, the extension of our results to these models is almost standard \cite{fl}.
We also believe that our method is robust enough fitting other models such as
independent random walks, the zero-range process, the generalized exclusion process, when a finite number of slow bonds is present.

The present work is divided as follows.
In Section \ref{s2}, we introduce notation and state the main result, namely Theorem \ref{t1}. In
Section \ref{s3} we make precise the scaling limit and sketch the proof of Theorem \ref{t1}. In
Section \ref{s4}, we prove tightness for any range of the parameter $\beta$.
In Section \ref{s5}, we prove the {\em{Replacement Lemma}} and we establish the energy estimates, which are fundamental for
characterizing the limit points and the uniqueness of weak solutions of the partial differential equations considered here. In Section \ref{s6} we characterize the limit points as weak solutions of the corresponding partial differential
equations. Finally, uniqueness of weak solutions is refereed to Section \ref{s7}.

\section{Notation and Results}\label{s2}
Let $\bb T_N=\{1,\ldots,N\}$ be the one-dimensional discrete torus with $N$ points. At each site, we allow at most one particle.
Therefore, we will be concerned about the state space $\{0,1\}^{\bb T_N}$. Configurations will be denoted by the Greek letter $\eta$,
so that $\eta(x)=1$, if the site $x$ is occupied, otherwise $\eta(x)=0$.\\

We define now the exclusion process with state space $\{0,1\}^{\bb T_N}$ and with conductance $\{\xi^{N}_{x,x+1}\}_x$ at the bond of vertices $x,x+1$.
The dynamics of this Markov process can be described as follows. To each bond of vertices $x,x+1$, we associate
an exponential clock
of parameter $\xi^{N}_{x,x+1}$. When this clock rings, the value of $\eta$ at the vertices of this bond are exchanged.
This process can also be characterized in terms of its infinitesimal generator $\mathcal{L}_{N}$, which
acts on local functions $f:\{0,1\}^{\bb T_N}\rightarrow \bb{R}$ as

\begin{equation*}
\mathcal{L}_{N}f(\eta)=\sum_{x\in \bb T_N}\,\xi^{N}_{x,x+1}\,\Big[f(\eta^{x,x+1})-f(\eta)\Big]\,,
\end{equation*}
where $\eta^{x,x+1}$ is the configuration obtained from $\eta$ by exchanging the variables $\eta(x)$ and $\eta(x+1)$:
\begin{equation*}
(\eta^{x,x+1})(y)=\left\{\begin{array}{cl}
\eta(x+1),& \mbox{if}\,\,\, y=x\,,\\
\eta(x),& \mbox{if} \,\,\,y=x+1\,,\\
\eta(y),& \mbox{otherwise.}
\end{array}
\right.
\end{equation*}\\

The Bernoulli product measures $\{\nu^N_\alpha : 0\le \alpha \le 1\}$ are invariant and in fact,
reversible, for the dynamics introduced above. Namely, $\nu^N_\alpha$ is a product
measure on $\{0,1\}^{\bb T_N}$ with marginal at site $x$ in $\bb T_N$ given by
\begin{equation*}
\nu^N_\alpha \{\eta : \eta(x) =1\} \;=\; \alpha.
\end{equation*}

Denote by $\bb T$ the one-dimensional continuous torus $[0,1)$.
The exclusion process with a slow bond at each point $b_1\ldots,b_k\in\bb T$ is defined with the following conductances:

\begin{equation*}
\xi^{N}_{x,x+1}\;=\;\left\{\begin{array}{cl}
N^{-\beta}, &  \mbox{if}\,\,\,\,\{b_1,\ldots,b_k\}\cap (\frac{x}{N},\frac{x+1}{N}]\neq \varnothing\,,\\
\\
1, &\mbox{otherwise\,.}
\end{array}
\right.
\end{equation*}

The conductances are chosen in such a way that particles cross bonds at rate one,
except $k$ particular bonds in which the dynamics is slowed down by a factor $N^{-\beta}$, with $\beta\in{[0,\infty)}$.
Each one of these particular bonds contains the macroscopic point $b_i\in\bb T$; or $b_i$ coincides with
some vertex $\pfrac{x}{N}$ and the slow bond is chosen as the bond to the left of $\pfrac{x}{N}$. To simplify notation,
we denote by $Nb_i$ the left vertex of the slow bond containing $b_i$. \\

Denote by $\{\eta_t:=\eta_{tN^2}: t\ge 0\}$ the Markov process on $\{0,1\}^{\bb T_N}$
associated to the generator $\mathcal{L}_N$ \emph{speeded up} by
$N^2$. Although $\eta_t$ depends on $N$ and $\beta$, we are not indexing it on that in order not to overload notation.
Let $D(\bb R_+, \{0,1\}^{\bb T_N})$ be the path space of
c\`adl\`ag trajectories with values in $\{0,1\}^{\bb T_N}$. For a
measure $\mu_N$ on $\{0,1\}^{\bb T_N}$, denote by $\bb P_{\mu_N}^\beta$ the
probability measure on $D(\bb R_+, \{0,1\}^{\bb T_N})$ induced by the
initial state $\mu_N$ and the Markov process $\{\eta_t : t\ge 0\}$ and denote by  $\bb E_{\mu_N}^\beta$
 the expectation with respect to $\bb P_{\mu_N}^\beta$.\\

\begin{definition} \label{def associated measures}
A sequence of probability measures $\{\mu_N : N\geq 1 \}$ on $\{0,1\}^{\bb T_N}$ is
said to be associated to a profile $\rho_0 :\bb T \to [0,1]$ if for every $\delta>0$ and every continuous functions $H: \bb T \to \bb R$
\begin{equation}\label{associated}
\lim_{N\to\infty}
\mu_N \Big\{ \eta:\, \Big\vert \pfrac 1N \sum_{x\in\bb T_N} H(\pfrac{x}{N})\, \eta(x)
- \int_{\bb{T}} H(u)\, \rho_0(u) du \Big\vert > \delta \Big\} \;=\; 0.
\end{equation}
\end{definition}
Now we introduce an operator which
corresponds to the generator of the random walk in $\bb T_N$ with conductance $\xi^N_{x,x+1}$ at the bond of vertices $x,x+1$. This operator acts on $H:\mathbb{T} \rightarrow \mathbb{R}$ as
\begin{equation}\label{generator random walk}
\mathbb{L}_N H(\pfrac{x}{N}) =  \xi^N_{x,x+1} \, \Big[ H\Big(\pfrac{x+1}{N}\Big)
- H\Big(\pfrac{x}{N}\Big) \Big] + \xi^N_{x-1,x} \, \Big[H\Big(\pfrac{x-1}{N}\Big) - H\Big(\pfrac{x}{N}\Big) \Big] \,.
\end{equation}

 We will not
differentiate the notation for functions $H$ defined on $\bb{T}$ and on $ \bb T_N$.
The indicator function of a set $A$ will be written by $\textbf 1_{A}(u)$, which is
one when $u\in A$ and zero otherwise.


\subsection{The Operator $\frac{d}{dx}\frac{d}{dW}$}

\label{weak operator}

\quad
\vspace{0.2cm}

Given the points $b_1,\ldots,b_k\in \bb T$, define the measure $W(du)$ in the torus $\bb T$ by
\begin{equation*}
 W(du)\;=\; du+ \delta_{b_1}(du)+\cdots+\delta_{b_k}(du)\,,
\end{equation*}
so that $W$ is the Lebesgue measure on the torus $\bb{T}$ plus the sum of the Dirac measure in each of the $\{b_i:i=1,...,k\}$.

Let $\mc H^1_W$ be the set of functions $F$ in $L^2(\bb T)$ such that for $x\in{\bb T}$
\begin{equation*}
F(x) \;=\; a \;+\; \int_{(0,x]}\Big(b+\int_0^y f(z) \, dz\Big) W(dy) ,
\end{equation*}
for some function $ f$ in $L^2(\bb T)$ and $a,b\in{\mathbb{R}}$ such that
\begin{equation}\label{domain}
\int_0^1  f(x) \, dx \;=\; 0\;, \quad
\int_{(0,1]}  \Big( b + \int_0^y f(z) \, dz \Big) W(dy)\;=\;0\; .
\end{equation}
Define the operator
\begin{equation*}
\begin{split}
&\frac{d}{dx}\frac{d}{dW} : \mc H^1_W \to L^2(\bb T)\\
&\frac{d}{dx}\frac{d}{dW} F =  f.
\end{split}
\end{equation*}
For more details we refer the reader to \cite{fl}.
\subsection{The hydrodynamical equations}

\quad
\vspace{0.2cm}

Consider a continuous density profile $\gamma : \bb T \to  [0,1]$.
Denote by $\<\cdot,\cdot\>$ the inner product in $L^2(\mathbb{T})$, by $ \rho_t$  a function $\rho(t,\cdot)$ and for an integer $n$ denote
by $C^n(\mathbb{T})$ the set of continuous functions from $\mathbb{T}$ to $\mathbb{R}$ and with continuous derivatives
of order up to $n$.
For $\mathcal I$ an interval of $\bb{T}$, here and in the sequel, for $n$ and $m$ integers, we use the notation $C^{n,m}([0,T]\times \mathcal I)$ to denote the set of functions defined on the domain $[0,T]\times \mathcal I$, that are of class $C^n$ in time and
$C^m$ in space.\\

\begin{definition}
A bounded function $\rho : [0,T] \times \bb T \to \bb R$
 is said to be a weak solution of the parabolic differential equation with initial condition $\gamma(\cdot)$:
\begin{equation}\label{edp1}
\left\{
\begin{array}{l}
{\displaystyle \partial_t \rho \; =\; \partial_u^2 \rho } \\
{\displaystyle \rho(0,\cdot) \;=\; \gamma(\cdot)}
\end{array}
\right.
\end{equation}
if, for $t\in{[0,T]}$ and $H\in C^2(\bb T) $, $\rho(t,\cdot)$ satisfies the integral equation
\begin{equation*}
\< \rho_t, H\> \;-\; \< \gamma , H\>
\;-\; \int_0^t \< \rho_s , \partial_u^2  H \>\, ds\;=0.
\end{equation*}
\end{definition}

\begin{definition}\label{def weak solution edp2}
A bounded function $\rho : [0,T] \times \bb T \to \bb R$
 is said to be a weak solution of the parabolic differential equation with initial condition $\gamma(\cdot)$:
\begin{equation}\label{edp2}
\left\{
\begin{array}{l}
{\displaystyle \partial_t \rho \; =\; \frac{d}{dx}\frac{d}{dW} \rho } \\
{\displaystyle \rho(0,\cdot) \;=\; \gamma(\cdot)}
\end{array}
\right.
\end{equation}
if, for $t\in{[0,T]}$ and $H\in\mc H^1_W$, $\rho(t,\cdot)$ satisfies the integral equation
\begin{equation*}
\< \rho_t, H\> \;-\; \< \gamma , H\>
 -\int_0^t \Big\< \rho_s , \frac{d}{dx}\frac{d}{dW}  H \Big\>\, ds\;=\;0\;.
\end{equation*}
\end{definition}

Following the notation of \cite{e}, denote by $L^2(0,T;\mc H^1(a,b))$ the space of functions
$\varrho\in L^2([0,T]\times [a,b])$ for which there exists a function in $L^2([0,T]\times [a,b])$,
denoted by $\p_u\varrho $, satisfying
\begin{equation*}
\int_0^T \int_a^b  \, (\partial_u H) (s, u)\, \varrho(s,u)\hspace{0.1cm}du\hspace{0.1cm}ds
\;=\; -\; \int_0^T \int_a^b \,   H (s, u)\, (\partial_u \varrho) (s, u)\hspace{0.1cm}du\hspace{0.1cm}ds\,,
\end{equation*}
for any $H\in C^{0,1}([0,T]\times [a,b])$ with compact support in
$[0,T]\times (a,b)$.

\begin{definition}\label{def eq neumman cond}
Let $[b_i,b_{i+1}]\subset\bb T$.
A bounded function $\rho : [0,T] \times [b_i,b_{i+1}] \to \bb R$
 is said to be a weak solution of the parabolic differential equation
with Neumann's boundary conditions in the cylinder $[0,T]\times [b_i,b_{i+1}]$ and with initial condition $\gamma(\cdot)$:
\begin{equation}\label{edp3}
\left\{
\begin{array}{l}
{\displaystyle \partial_t \rho \; =\; \partial_u^2 \rho }\\
{\displaystyle \rho(0,\cdot) \;=\; \gamma(\cdot)}\\
{\displaystyle\partial_u\rho(t,b_i)=\displaystyle\partial_u\rho(t,b_{i+1})}=0,\;\forall t\in[0,T]
\end{array}
\right.
\end{equation}
 if, for $t\in{[0,T]}$ and $H\in C^{1,2}([0,T]\times[b_i,b_{i+1}])$, $\rho(t,\cdot)$ satisfies the integral equation
\begin{equation}\label{int}
\begin{split}
&\int_{b_i}^{b_{i+1}} \rho(t,u)\,H(t,u)\,du -\int_{b_i}^{b_{i+1}}\gamma(u)\,H(0,u)\,du\\
-& \int_0^t\int_{b_i}^{b_{i+1}} \rho(s,u) \, \{\partial_u^2  H(s,u)+\partial_s H(s,u)\}\,du\, ds\\
+&\;\int_0^t\partial_u H (s,b_{i+1})\,\rho(s,b_{i+i}^-)\,ds-\int_0^t\partial_uH(s,b_i)\,\rho(s,b_i^+)\,ds\;=\;0\;
\end{split}
\end{equation}
and $\rho(t,\cdot)$ belongs to $L^2(0,T;\mc H^1(b_i,b_{i+1}))$.
\end{definition}

Since in Definition \ref{def eq neumman cond} we impose $\rho\in L^2(0,T;\mc H^1(b_i,b_{i+1}))$,
the integrals are well-defined at the boundary. This is a consequence of the following two facts. On one hand, it follows from the assumption that $\rho(t,\cdot)\in \mathcal{H}^1(b_i,b_{i+1})$, almost surely in $t\in [0,T]$.
On the other hand, it is well-known that functions belonging to $\mathcal{H}^1(b_i,b_{i+1})$ and with sided limits at $b_i$ and
$b_{i+1}$ are absolutely continuous with respect to the Lebesgue measure, see \cite{l} for instance. We refer the reader to
\cite{e} for classical results about Sobolev spaces.

Heuristically, in order to establish an integral equation for the
weak solution of the heat equation with Neumann's boundary conditions as above, one should multiply
 \eqref{edp3} by a test function $H$ and perform twice a formal integration by parts to arrive
at \eqref{int}. \\

 We are now in position to state the main result of this paper:
 \begin{theorem}\label{t1}
Fix $\beta\in [0,\infty)$. Consider the exclusion process with $k$ slow bonds corresponding to macroscopic points
 $b_1,\ldots,b_k\in{\bb{T}}$ and with conductance $N^{-\beta}$ at each one of these slow bonds.

Fix a continuous initial profile $\gamma : \bb T \to [0,1]$.  Let $\{\mu_N: N\geq{1}\}$ be
a sequence of probability measures  on $\{0,1\}^{\bb T_N}$ associated to $\gamma$.
 Then, for any $t\in [0,T]$, for every $\delta>0$ and every $H\in C(\bb{T})$, it holds that
\begin{equation*}
\lim_{N\to\infty}
\bb P_{\mu_N}^\beta \Big\{\eta_. : \, \Big\vert \pfrac{1}{N} \sum_{x\in\bb{T}_N}
H(\pfrac{x}{N})\, \eta_t(x) - \int_{\bb T} H(u)\, \rho(t,u) du \Big\vert
> \delta \Big\} \;=\; 0\,,
\end{equation*}
 where :
\begin{itemize}
\item
if $\beta\in[0,1)$, $\rho(t,\cdot)$ is the unique weak solution of \eqref{edp1};
\item
if $\beta=1$, $\rho(t,\cdot)$ is the unique weak solution of \eqref{edp2};
\item
 if $\beta\in(1,\infty)$, in each cylinder $[0,T]\times[b_i,b_{i+1}]$, $\rho(t,\cdot)$ is the unique weak solution of \eqref{edp3}.
\end{itemize}
\end{theorem}

\begin{remark}
 The assumption that all slow bonds have exactly the same conductance is not necessary at all. In fact, last result is true when considering each slow bond containing the macroscopic point $b_i$ with conductance $N^{-\beta_i}$. In that case, we would obtain a parabolic differential equation with the behavior at each $[b_i,b_{i+1}]$ given by the regime of the corresponding $\beta_i$ as above. Another straightforward generalization
is to consider conductances not exactly equal to $N^{-\beta}$, but of order $N^{-\beta}$, in the sense that the quotient with
$N^{-\beta}$ converges to one. For sake of clarity, we present the proof under the conditions of Theorem \ref{t1}.
\end{remark}

\section{Scaling Limit}\label{s3}

Let $\mc M$ be the space of positive measures on $\bb T$ with total
mass bounded by one, endowed with the weak topology. Let
$\pi^{N}_{t} \in \mc M$ be the empirical measure at time $t$ associated to $\eta_t$, namely,
it is the measure on $\bb T$ obtained by rescaling space by $N$ and
by assigning mass $N^{-1}$ to each particle:
\begin{equation}\label{f01}
\pi^{N}_{t} \;=\; \pfrac{1}{N} \sum _{x\in \bb T_N} \eta_t (x)\,
\delta_{x/N}\,,
\end{equation}
where $\delta_u$ is the Dirac measure concentrated on $u$.
For an integrable function
$H:\bb T \to \bb R$, $\<\pi^N_t, H\>$ stands for
the integral of $H$ with respect to $\pi^N_t$:
\begin{equation*}
\<\pi^N_t, H\> \;=\; \pfrac 1N \sum_{x\in\bb T_N}
H (\pfrac{x}{N})\, \eta_t(x)\,.
\end{equation*}
This notation is not to be mistaken with the inner product in $L^2(\bb R)$. Also, when $\pi_t$ has a density
$\rho$, namely when $\pi(t,du) = \rho(t,u) du$, we sometimes write $\<\rho_t, H\>$
for $\<\pi_t, H\>$.\\

Fix $T>0$. Let $D([0,T], \mc M)$ be the space of $\mc M$-valued
c\`adl\`ag trajectories $\pi:[0,T]\to\mc M$ endowed with the
\emph{Skorohod} topology.  For each probability measure $\mu_N$ on
$\{0,1\}^{\bb T_N}$, denote by $\bb Q_{\mu_N}^{\beta,N}$ the measure on
the path space $D([0,T], \mc M)$ induced by the measure $\mu_N$ and
the empirical process $\pi^N_t$ introduced in \eqref{f01}.

Fix a continuous profile $\gamma : \bb T \to [0,1]$ and consider a
sequence $\{\mu_N : N\ge 1\}$ of measures on $\{0,1\}^{\bb T_N}$
associated to $\gamma$. Let $\bb Q^{\beta}$ be
the probability measure on $D([0,T], \mc M)$ concentrated on the
deterministic path $\pi(t,du) = \rho (t,u)du$, where:

\begin{itemize}
\item
if $\beta\in[0,1)$, $\rho(t,\cdot)$ is the unique weak solution of \eqref{edp1};
\item
if $\beta=1$, $\rho(t,\cdot)$ is the unique weak solution of \eqref{edp2};
\item
 if $\beta\in(1,\infty)$, in each cylinder $[0,T]\times[b_i,b_{i+1}]$, $\rho(t,\cdot)$ is the unique weak solution of \eqref{edp3}.
\end{itemize}
\begin{proposition}
\label{s15}
As $N\uparrow\infty$, the sequence of probability measures $\{\bb
Q_{\mu_N}^{\beta,N}:N\geq{1}\}$ converges weakly to $\bb Q^{\beta}$.
\end{proposition}

The proof of this result is divided into three parts. In the next section,
we show that the sequence $\{\bb Q_{\mu_N}^{\beta,N} : N\ge
1\}$ is tight, for any $\beta\in [0,\infty)$. In Section \ref{s6} we characterize the limit
points of this sequence for each regime of the parameter $\beta$. Uniqueness of weak solutions is presented in Section
\ref{s7} and this implies the uniqueness of limit points of the sequence $\{\bb Q_{\mu_N}^{\beta,N} : N\ge
1\}$. In the fifth section, we prove a
suitable {\em{Replacement Lemma}} for each regime of $\beta$, which is crucial in the task of characterizing limit points and uniqueness.

\section{Tightness} \label{s4}
\begin{proposition}\label{s06}
For any fixed $\beta\in[0,\infty)$, the sequence of measures $\{\bb Q_{\mu_N}^{\beta,N} : N\ge 1\}$ is tight in
the Skorohod topology of $D([0,T],\mc M)$.
\end{proposition}

\begin{proof}
In order to prove
tightness of $\{\pi^{N}_t : 0\le t \le T\}$ it is enough to
show tightness of the real-valued processes $\{\<\pi^{N}_t ,H\> :
0\le t \le T\}$ for $H\in{C(\bb{T})}$. In fact, c.f. \cite{kl} it is enough to show tightness of $\{\<\pi^{N}_t ,H\> :
0\le t \le T\}$ for a dense set of functions in $C(\bb{T})$ with respect to the uniform topology. For that purpose, fix $H\in C^2(\bb T)$.
By Dynkin's formula,
\begin{equation}\label{M}
M^{N}_{t}(H)=\<\pi^{N}_{t}, H\>- \<\pi^{N}_{0}, H\>-\int_{0}^{t}N^2\mc L_{N}\<\pi^{N}_{s},H\>\,ds\,,
\end{equation}
is a martingale with respect to the natural filtration $\mathcal{F}_t:=\sigma(\eta_s: s\leq{t})$. In order to prove tightness of $\{\<\pi^{N}_t ,H\>:N\geq{1}\}$, we prove tightness
of the sequence of the martingales and the integral terms in the decomposition above. We start by the former.

We begin by showing that the $L^2(\mathbb{P}_{\mu_N}^\beta)$-norm of the martingale above vanishes as $N\rightarrow{+\infty}$.
The quadratic variation of $M^{N}_{t}(H)$  is given by
\begin{eqnarray}\label{varquad}
\!\!\!\!\!\!\!\!\<M^{N}(H)\>_t& = & \int_{0}^{t} \sum_{x\in\bb T_{N}} \xi^{N}_{x,x+1}\Big[(\eta_{s}(x)-\eta_{s}(x+1))
(H(\pfrac{x+1}{N})-H(\pfrac{x}{N}))\Big]^2ds.
\end{eqnarray}
It is easy to show that
$\<M^{N}(H)\>_t \;\leq\;   \pfrac{T}{N}\|\partial_u H\|_{\infty}^2$. Here and in the sequel we use the notation $\|H\|_{\infty}:=\sup_{u\in{\bb T}}|H(u)|.$

Thus,  $M^{N}_{t}(H)$ converges
to zero as $N\rightarrow{+\infty}$ in $L^2(\mathbb{P}_{\mu_N}^\beta)$. Notice that above we used the trivial bound $\xi^N_{x,x+1}\leq 1$. By Doob's inequality, for every $\delta>0$,
\begin{equation}\label{limprob}
\lim_{N\rightarrow\infty}\bb P_{\mu_N}^\beta\left[\sup_{0\leq t\leq T} |M^{N}_{t}(H)|>\delta\right]=0\,,
\end{equation}
which implies tightness of the sequence of martingales $\{M^{N}_{t}(H); N\geq 1\}$.
Now, we need to examine tightness of the integral term in (\ref{M}).

 Denote by $\Gamma_N$ the subset of sites $x\in \mathbb{T}_N$ such that $x$ has some adjacent slow bond, namely,
$\xi^N_{x,x+1}=N^{-\beta}$ or  $\xi^N_{x-1,x}=N^{-\beta}$.
The term $N^{2}\mc L_{N}\<\pi^{N}_{s},H\>$ appearing inside the time integral in (\ref{M}) is explicitly given by
\begin{equation*}
\begin{split}
 & N\sum_{x\notin\Gamma_N}
\eta_{s}(x)\Big[H(\pfrac{x+1}{N})+H(\pfrac{x-1}{N})-2H(\pfrac{x}{N})\Big]\\
+ \,& N\sum_{x\in\Gamma_N}
\eta_{s}(x)\Big[\xi^{N}_{x,x+1}\{H(\pfrac{x+1}{N})-H(\pfrac{x}{N})\}+\xi^N_{x-1,x}
\{H(\pfrac{x-1}{N})-H(\pfrac{x}{N})\}\Big]\,.
\end{split}
\end{equation*}
By Taylor expansion on $H$, the absolute value of the first sum above is bounded by
$\Vert\partial_u^2 H\Vert_\infty$. Since there are at most $2k$ elements in $\Gamma_N$,
$\xi_{x,x+1}\leq 1$ and since there is only one particle per site, the absolute value of the
second sum above is bounded by $2\,k\| \partial_u H\|_{\infty}$.
Therefore, there exists a constant $C:=C(H,k)>0$, such that
$|N^2\mc L_{N}\<\pi^{N}_{s},H\>|$ $\leq C$, which yields
\begin{equation*}
 \left|\int_{r}^{t}N^2\mc L_{N}\<\pi^{N}_{s},H\>ds\right|\leq C|t-r|\,.
\end{equation*}
By Proposition 4.1.6 of \cite{kl}, last inequality implies tightness of the integral term. This concludes the proof.
\end{proof}

\section{Replacement Lemma and Energy Estimates}\label{s5}
In this section, we obtain fundamental results that allow us to replace the mean occupation of a
site by the mean density of particles in a small macroscopic box around this site. This result implies that the limit trajectories must belong to  some Sobolev space, this will be clear later. Before proceeding we introduce some tools that we use in the sequel.  \\

Denote by $H_N (\mu_N | \nu_\alpha)$ the entropy of a probability
measure $\mu_N$ with respect to the invariant state $\nu_\alpha$. For a precise definition and properties of the entropy, we
refer the reader to \cite{kl}. In Proposition \ref{K0} in the Appendix we review a classical result saying that there exists a finite constant $K_0:=K_0(\alpha)$, such that
\begin{equation}
\label{f06}
H_N (\mu_N | \nu_\alpha) \;\le\; K_0 N,
\end{equation}
for any probability measure $\mu_N\in{\{0,1\}^{\mathbb{T}_N}}$.

Denote by $\< \cdot, \cdot \>_{\nu_\alpha}$ the scalar product of
$L^2(\nu_\alpha)$  and denote by $\mf D_N$  the Dirichlet form, which is the convex and lower
semicontinuous functional (see Corollary A1.10.3 of \cite{kl}) defined as:
\begin{equation*}
\mf D_N (f) \;=\; \< - L_N \sqrt f \,,\, \sqrt f\>_{\nu_\alpha},
\end{equation*}
where $f$ is a probability density with respect to $\nu_\alpha$ (i.e.
$f\ge 0$ and $\int f d\nu_\alpha =1$). An elementary computation shows
that
\begin{eqnarray*}
\!\!\!\!\!\!\!\!\!\!\!\!\!\! &&
\mf D_N (f) \;=\; \sum_{x\in \bb T_N} \frac{ \xi_{x,x+1}^N}{2}
\int  \Big( \sqrt{f(\eta^{x,x+1})} -
\sqrt{f(\eta)} \Big)^2 \, d\nu_\alpha \;.
\end{eqnarray*}
By Theorem A1.9.2 of \cite{kl}, if $\{S^N_t : t\ge 0\}$ stands for the
semi-group associated to the generator $N^2\mathcal{L}_N$, then
\begin{equation*}
H_N (\mu_N S^N_t | \nu_\alpha) \; +\; N^2 \, \int_0^t
\mf D_N (f^N_s) \, ds  \;\le\; H_N (\mu_N | \nu_\alpha)\;,
\end{equation*}
provided $f^N_s$ stands for the Radon-Nikodym derivative of $\mu_N
S^N_s$ (the distribution of $\eta_s$ starting from $\mu_N$) with respect to $\nu_\alpha$. \\

\subsection{Replacement Lemma}

\quad
\vspace{0.2cm}

Now, we define the local density of particles,
which corresponds to the mean occupation in a box around a given site. We represent this empirical density in the box of size $\ell$ around a given site $x$ by $\eta^{\ell}(x)$.
For $\beta\in[0,1)$, this
box can be chosen in the usual way, but for $\beta\in[1,\infty)$, this box must avoid the slow bond. From this point on, we denote the integer part of $\eps N$, namely $\lfloor \eps N\rfloor$, simply by $\eps N$.
\begin{definition}\label{def4}
For $\beta\in[0,1)$, define the empirical density by
\begin{equation*}
 \eta^{\eps N}(x)\;=\;\pfrac{1}{\eps N}\sum_{y=x+1}^{x+\eps N}\eta(y)\,.
\end{equation*}
\end{definition}
\begin{definition}\label{def5}
For $\beta\in[1,\infty)$, if $x$ is such that
 $\{Nb_1,\ldots,N b_k\}\cap \{x,\ldots,x+\eps N\}=\varnothing$, then the empirical density
is defined by
\begin{equation*}
 \eta^{\eps N}(x)\;=\;\pfrac{1}{\eps N}\sum_{y=x+1}^{x+\eps N}\eta(y)\,.
\end{equation*}
Otherwise, if, let us say, $Nb_i\in \{x,\ldots,x+\eps N\}$ for some $i=1,..,k$, then the empirical density is defined by
\begin{equation*}
 \eta^{\eps N}(x)\;=\;\pfrac{1}{\eps N}\sum_{y=Nb_i-\eps N+1}^{Nb_i}
\eta(y)\,.
\end{equation*}

Since we are considering a finite number of slow bonds, the distance between two consecutive macroscopic points related to two consecutive slow bonds is at least $\eps$, for $\eps$ sufficiently small. As a consequence, we can suppose, without lost of generality that in the previous definition, $b_i$ is unique.

\end{definition}
\begin{lemma}\label{lema41}
Fix $\beta\in[0,1)$. Let $f$ be a density with respect to the invariant measure $\nu_\alpha$. Then,
\begin{equation*}
 \int \{\eta(x)-\eta^{\eps N}(x)\}f(\eta)\nu_\alpha(d\eta)\;\leq\; 2(kN^{\beta-1}+\eps)+N\,\mf D_N(f)\,, \forall{x\in{\mathbb{T}_N}}.
\end{equation*}
\end{lemma}
 \begin{proof}
From Definition \ref{def4} we have that
 \begin{equation*}
 \int \{\eta(x)-\eta^{\eps N}(x)\}f(\eta)\nu_\alpha(d\eta)=
 \int \Big\{\pfrac{1}{\eps N}\sum_{y=x+1}^{x+\eps N}(\eta(x)-\eta(y))\Big\}f(\eta)\,\nu_\alpha(d\eta)\,.
 \end{equation*}
Writing $\eta(x)-\eta(y)$ as a telescopic sum, the last expression becomes equal to
 \begin{equation*}
  \int \Big\{\pfrac{1}{\eps N}\sum_{y=x+1}^{x+\eps N}
\sum_{z=x}^{y-1}(\eta(z)-\eta(z+1))\Big\}f(\eta)\,\nu_\alpha(d\eta)\,.
 \end{equation*}
Rewriting the expression above as twice the half and making the transformation $\eta\mapsto \eta^{z,z+1}$ (for
which the probability $\nu_\alpha$ is invariant) it becomes as:
\begin{equation*}
 \pfrac{1}{2\eps N}\sum_{y=x+1}^{x+\eps N}
\sum_{z=x}^{y-1}\int \{ \eta(z)-\eta(z+1)\}(f(\eta)-f(\eta^{z,z+1}))\,\nu_\alpha(d\eta)\,.
 \end{equation*}
Since $(a-b)=(\sqrt{a}-\sqrt{b})(\sqrt{a}+\sqrt{b})$ and by the Cauchy-Schwarz's inequality, for any $A>0$, we bound the previous expression from above by
\begin{equation*}
\begin{split}
 &\pfrac{1}{2\eps N}\sum_{y=x+1}^{x+\eps N}
\sum_{z=x}^{y-1}\frac{A}{\xi_{z,z+1}^N}\int\{\eta(z)-\eta(z+1)\}^2\Big(\sqrt{f(\eta)}+\sqrt{f(\eta^{z,z+1})}\Big)^2\,\nu_\alpha(d\eta)\\\
 +\,&\pfrac{1}{2\eps N}\sum_{y=x+1}^{x+\eps N}
\sum_{z=x}^{y-1}\frac{\xi_{z,z+1}^N}{A}\int \Big(\sqrt{f(\eta)}-\sqrt{f(\eta^{z,z+1})}\Big)^ 2\,\nu_\alpha(d\eta)\,.
\end{split}
 \end{equation*}

The second sum above is bounded by
\begin{equation*}
\pfrac{1}{2\eps N}\sum_{y=x+1}^{x+\eps N} \sum_{z\in{\bb{T}_N}}\frac{\xi_{z,z+1}^N}{A}\int \Big(\sqrt{f(\eta)}-\sqrt{f(\eta^{z,z+1})}\Big)^ 2\nu_\alpha(d\eta)\\
=\pfrac{1}{A}\mf D_N(f)\,.
 \end{equation*}
 On the other hand, since $f$ is a density,
the first sum
is bounded from above by
\begin{equation*}
 \pfrac{1}{2\eps N}\sum_{y=x+1}^{x+\eps N} \sum_{z=x}^{y-1}\frac{4A}{\xi_{z,z+1}^N}\leq
\pfrac{1}{\eps N}\sum_{y=x+1}^{x+\eps N}2A(kN^\beta+\eps N)=2A(kN^\beta+\eps N)\,.
\end{equation*}
Notice that the term $kN^{\beta}$ comes from the existence of $k$ slow bonds. Choosing $A=\pfrac{1}{N}$, the proof ends.
 \end{proof}

\begin{lemma}[Replacement Lemma]\label{replace1}
\quad

Fix $\beta\in[0,1)$.
 Let $b\in \bb T$ and let $x$ be the right (or left) vertex of the bond containing the macroscopic point $b$. Then,
\begin{equation*}
 \varlimsup_{\eps\to 0}\varlimsup_{N\to\infty}
\bb E_{\mu_N}^\beta\Big[\,\Big|\int_0^t\{\eta_s(x)-\eta_s^{\eps N}(x)\}\, ds\Big|\,\Big]\;=\;0\,.
\end{equation*}
\end{lemma}
\begin{proof}
From Jensen's inequality together with the entropy inequality (see for example Appendix 1 of \cite{kl}), for any $\gamma\in\bb R$ (which will be chosen large),
the expectation appearing on the statement of the Lemma is bounded from above by
\begin{equation}\label{log}
\frac{H_N(\mu_N|\nu_\alpha)}{\gamma N}+\frac{1}{\gamma N}\log \bb E_{\nu_\alpha} \Big[
\exp\Big\{\gamma\,N\Big|\int_0^t \{
\eta_s(x)-\eta^{\eps N}_s(x)\}\,ds\Big|\Big\}\Big]\,.
\end{equation}
By Proposition \ref{K0}, $H_N(\mu_N|\nu_\alpha)\leq K_0\,N$, so that it remains to focus on the second summand
above. Since $e^{|x|}\leq e^x+e^{-x}$ and
\begin{equation}\label{log bounds}
\varlimsup_N \pfrac{1}{N}\log (a_N+b_N)= \max\Big\{
\varlimsup_N \pfrac{1}{N}\log a_N,\varlimsup_N \pfrac{1}{N}\log b_N\Big\}\,,
\end{equation}
we can remove the modulus inside the exponential. By Feynman-Kac's formula, see Lemma A1.7.2 of \cite{kl} and Proposition
\ref{densidade},
 the second term on the right hand side of \eqref{log} is less than or equal to
\begin{equation*}
t\sup_{f\textrm{~density}}\Big\{\int
\{\eta(x)-\eta^{\eps N}(x)\}f(\eta)\nu_\alpha(d\eta) - N\,\mf D_N(f)\Big\}\,.
\end{equation*}
Applying Lemma \ref{lema41} and recalling that $\gamma$ is arbitrarily large, the proof finishes.
\end{proof}

The next two results are concerned with both cases $\beta=1$ and $\beta\in(1,\infty)$.

\begin{lemma}\label{lemma43} Fix $\beta\in[1,\infty)$. Let $f$ be a density with respect to the invariant measure $\nu_\alpha$.
Then,
\begin{equation*}
\int  \{\eta(x)-\eta^{\eps N}(x)\}f(\eta)\nu_\alpha(d\eta)
\leq N\mf D_N(f)+4\eps\,, \forall{x\in{\mathbb{T}_N}}.
\end{equation*}
Moreover, given a function $H:\bb T\to\bb R$:
\begin{equation*}
\pfrac{1}{N} \sum_{x\in{\bb{T}_N}}\int  H(\pfrac{x}{N})\{\eta(x)-\eta^{\eps N}(x)\}f(\eta)\nu_\alpha(d\eta)
\leq N\mf D_N(f)+\pfrac{4\eps}{N}\sum_{x\in{\bb{T}_N}} \Big(H(\pfrac{x}{N})\Big)^2\,.
\end{equation*}
\end{lemma}
\begin{proof} Recall the Definition \ref{def5}.  Let first $x$ be a site
such that there is no slow bond connecting two sites in $\{x,\ldots,x+\eps N\}$. In this case,
\begin{equation*}
\begin{split}
&\int H(\pfrac{x}{N})\{\eta(x)-\eta^{\eps N}(x)\}f(\eta)\nu_\alpha(d\eta)\\
=\,&\int  H(\pfrac{x}{N})\Big\{\pfrac{1}{\eps N}\sum_{y=x+1}^{x+\eps N}
(\eta(x)-\eta(y))\Big\}f(\eta)\nu_\alpha(d\eta)\,,
\end{split}
\end{equation*}
and following the same arguments as in Lemma \ref{lema41},
we bound the previous expression from above by
\begin{equation*}
\begin{split}
&\pfrac{(H(\pfrac{x}{N}))^2}{2\eps N}\sum_{y=x+1}^{x+\eps N}
\sum_{z=x}^{y-1}\int\frac{A}{\xi^N_{z,z+1}}\{\eta(z)-\eta(z+1)\}^2\Big(\sqrt{f(\eta)}+\sqrt{f(\eta^{z,z+1})}\Big)^2\nu_\alpha(d\eta)\\
+&\frac{1}{2\eps N}\sum_{y=x+1}^{x+\eps N}
\sum_{z=x}^{y-1}\int \frac{\xi^N_{z,z+1}}{A}\{\eta(z)-\eta(z+1)\}^2\Big(\sqrt{f(\eta)}-\sqrt{f(\eta^{z,z+1})}\Big)^2 \nu_\alpha(d\eta)\,.
\end{split}
\end{equation*}
Since $\xi^N_{z,z+1}=1$ for all $z\in\{x,\ldots,x+\eps N-1\}$, it yields the boundedness of the previous expression by
\begin{equation*}
2\eps N A \,\Big(H(\pfrac{x}{N})\Big)^2+\frac{\mf D_N(f)}{A}\,.
\end{equation*}
Let now $x$ be a site such that $Nb_i\in\{x,\ldots,x+\eps N\}$ for some $i=1,\ldots,k$.
In this case,
\begin{equation}\label{equality}
\begin{split}
&\int H(\pfrac{x}{N})\{\eta(x)-\eta^{\eps N}(x)\}f(\eta)\nu_\alpha(d\eta)\\
=\,&\int  H(\pfrac{x}{N})\frac{1}{\eps N}\sum_{y=Nb_i-\eps N+1}^{Nb_i} \Big\{\eta(x)-
\eta(y)\Big\}f(\eta)\nu_\alpha(d\eta)\,
\end{split}
\end{equation}
Now we split the last summation into two cases, $y>x$ and $y<x$ and then we proceed by writing $\eta(x)-\eta(y)$ as a telescopic sum as in Lemma \ref{lema41}.
Then, by the same arguments of Lemma \ref{lema41} and since $\xi^N_{z,z+1}=1$ for all $z$ in the range $\{Nb_i-\eps N+1,\ldots,Nb_i-1\}$,
we bound the previous expression by
\begin{equation*}
4\eps N A\Big(H(\pfrac{x}{N})\Big)^2+\frac{\mf D_N(f)}{A}\,.
\end{equation*}
Now the first claim of the lemma follows by taking the particular case
$H(\pfrac{x}{N})=1$ and choosing  $A=\pfrac{1}{N}$.

Finally, if in \eqref{equality} we sum over $x\in \bb{T}_N$ and then divide by $N$, one concludes
 the second claim of the lemma.
\end{proof}

\begin{lemma}[Replacement Lemma]\label{replace2}
\quad

Fix $\beta\in [1,\infty)$. Then, for every $x\in{\mathbb{T}_N}$
\begin{equation*}
\varlimsup_{\eps\to 0}\varlimsup_{N\to \infty}\bb E_{\mu_N}^\beta\Big[\,\Big|\int_0^t
\{\eta_s(x)-\eta^{\eps N}_s(x)\}\,ds\,\Big|\,\Big]\;=\;0\,.
\end{equation*}
Moreover, given a function $H:\bb T\to \bb R$ satisfying
\begin{equation*}
 \varlimsup_{N\to\infty}\pfrac{1}{N}\sum_{x\in \bb{T}_N}\Big(H(\pfrac{x}{N})\Big)^2\;<\;\infty\,,
\end{equation*}
also holds
\begin{equation*}
\varlimsup_{\eps\to 0}\varlimsup_{N\to \infty}\bb E_{\mu_N}^\beta\Big[\,\Big|\int_0^t\pfrac{1}{N}\sum_{x\in{\bb{T}_N}} H(\pfrac{x}{N})\{
\eta_s(x)-\eta^{\eps N}_s(x)\}\,ds\,\Big|\,\Big]\;=\;0\,.
\end{equation*}
\end{lemma}
\begin{proof}
The proof follows exactly the same arguments in Lemma \ref{replace1}. Therefore,
 is sufficient to show that the expressions
\begin{equation*}
t\sup_{f\textrm{~density}}\Big\{\int \{\eta(x)-\eta^{\eps N}(x)\}f(\eta)d\nu_\alpha - N\mf D_N(f)\Big\}
\end{equation*}
and
\begin{equation*}
t\sup_{f\textrm{~density}}\Big\{\int \pfrac{1}{N}
\sum_x H(\pfrac{x}{N})\{\eta(x)-\eta^{\eps N}(x)\}f(\eta)d\nu_\alpha - N\mf D_N(f)\Big\}\,,
\end{equation*}
vanish as $N\rightarrow{+\infty}$, which is an immediate consequence of Lemma \ref{lemma43}.
\end{proof}
In the next subsection, we will need the following variation of Lemma \ref{lemma43}:
\begin{lemma}\label{lemma55}
 Let $H:\bb{T}\to\bb R$ and let $f$ be a density with respect to $\nu_\alpha$. Then, for every $x\in{\mathbb{T}_N}$
\begin{equation*}
\begin{split}
&\int \frac{1}{\eps N}\sum_{x\in \bb{T}_N} H(\pfrac{x}{N})\Big\{\eta(x)-\eta(x+\eps N)\Big\}f(\eta)\,\nu_\alpha(d\eta)\\
\leq \,&N\mf D_N(f)+\frac{2}{\eps N}\sum_{x\in \bb{T}_N}\Big(H(\pfrac{x}{N})\Big)^2
\Big\{\eps+N^{\beta-1}\sum_{i=1}^k\textbf{1}_{[b_i,b_i+\eps)}(\pfrac{x}{N})\Big\}\,.
\end{split}
\end{equation*}
\end{lemma}

The proof of the last lemma follows the same steps as above and for that reason will be omitted. Nevertheless, we sketch the idea of the proof. One begins by writing
$\eta(x)-\eta(x+\eps N)$ as a telescopic sum and proceeding as in Lemma \ref{lemma43}.
The only relevant difference in this case is that is not possible to avoid the slow bonds inside the telescopic sum, and
therefore the upper bound depends on $\beta$.

\subsection{Energy Estimates}

\quad
\vspace{0.2cm}

We prove in this subsection that any limit point $\bb Q_*^{\beta}$ of the
sequence $\{\bb Q_{\mu_N}^{\beta,N}:N\geq{1}\}$ is concentrated on trajectories
$\rho(t,u) du$ with finite energy, meaning that $\rho(t,u)$ belongs to some Sobolev space. For $\beta\in [0,1)$, this result is an immediate consequence of the
uniqueness of weak solutions of the heat equation. The case $\beta=1$ is a particular case of the one considered in \cite{fl}.
Therefore, we will treat here the remaining case $\beta\in(1,\infty)$. Such result will play an important role in the uniqueness
of weak solutions of \eqref{edp3}.\\

Let $\bb Q_*^{\beta}$ be a limit point of $\{\bb Q_{\mu_N}^{\beta,N}:N\geq{1}\}$ and assume without lost of generality that the whole
sequence converges weakly to $\bb Q_*^{\beta}$.

\begin{proposition}
\label{s05}
The measure $\bb Q_*^{\beta}$ is concentrated on paths $\pi(t,u)=\rho(t,u) du$. Moreover, there exists a function in $L^2([0,T]\times \bb T)$, denoted by $\partial_u \rho$, such that
\begin{equation*}
\int_0^T \int_{\bb T}  \, (\partial_u H) (s, u) \, \rho(s,u)\hspace{0.1cm}du \hspace{0.1cm}ds
\;=\; -\; \int_0^T \int_{\bb T} \,   H (s, u)\, (\partial_u \rho) (s, u)\hspace{0.1cm}du \hspace{0.1cm}ds\,,
\end{equation*}
for all $H$ in $C^{0,1}([0,T]\times \bb T)$ whose support is contained in
$[0,T]\times(\bb T\backslash \{b_1,\ldots,b_k\})$.
\end{proposition}

The previous result follows from the next lemma.  Recall the
definition of the constant $K_0$ given in \eqref{f06}.

\begin{lemma}
\label{s03}
\begin{equation*}
 \begin{split}
E_{\bb Q_*^{\beta}} \Big[ \sup_H \Big\{ \int_0^T \, \int_{\bb T}
 \, (\partial_u H) (s, u) \, &\rho(s,u)\hspace{0.1cm}du\hspace{0.1cm}ds\\
&- \; 2 \int_0^T\,\int_{\bb T}\, \Big(H (s, u)\Big)^2\hspace{0.1cm}du\hspace{0.1cm}ds\Big\} \Big] \; \le \; K_0 \;,
 \end{split}
\end{equation*}
where the supremum is carried over all functions $H$ in
$C^{0,1}([0,T]\times \bb T)$ with support contained in
$[0,T]\times (\bb T\backslash \{b_1,\ldots,b_k\})$.
\end{lemma}
We start by showing Proposition \ref{s05} assuming the last result. Later and independently we will prove the previous lemma.

\begin{proof}[Proof of Proposition \ref{s05}]
Denote by $\ell : C^{0,1}([0,T]\times \bb T) \to \bb R$ the linear
functional defined by
\begin{equation*}
\ell (H) \;=\; \int_0^T \, \int_{\bb T} \, (\partial_u H) (s, u) \, \rho(s,u)\hspace{0.1cm}du\hspace{0.1cm}ds.
\end{equation*}
Since the set of functions $H\in C^{0,1}([0,T]\times \bb T)$  with support
contained in $[0,T]\times (\bb T\backslash \{b_1,\ldots,b_k\})$
is dense in $L^2([0,T]\times \bb T)$ and since by Lemma \ref{s03}, $\ell$ is a $\bb Q_*^{\beta}$-a.s. bounded functional in $C^{0,1}([0,T]\times \bb T)$, we can extend it to a $\bb Q_*^{\beta}$-a.s. bounded functional in $L^2([0,T]\times \bb T)$. In particular, by the Riesz Representation
Theorem, there exists a function $G$ in $L^2([0,T]\times \bb T)$
such that
\begin{equation*}
\ell (H) \;=\; - \int_0^T \, \int_{\bb T} \, H (s, u) \, G(s,u)\hspace{0.1cm}du\hspace{0.1cm}ds\;.
\end{equation*}
This finishes the proof.
\end{proof}

For a smooth function $H\colon \bb T\to \bb R$, $\varepsilon
>0$ and a positive integer $N$, define $V_N(\varepsilon, H,\eta)$ by
\begin{eqnarray*}
V_N(\varepsilon, H , \eta ) &=&
\pfrac 1{\varepsilon N} \sum_{x\in\bb T_N} H(\pfrac{x}{N})
\{ \eta(x) -  \eta(x + \varepsilon N)\} -  \pfrac {2}{ N} \sum_{x\in\bb T_N}\Big(H(\pfrac{x}{N})\Big)^2 \; .
\end{eqnarray*}
In order to prove the Lemma \ref{s03}, we need the following technical result:

\begin{lemma}
Consider  $H_1,\ldots,H_k$  functions in
$C^{0,1}([0,T]\times \bb T)$ with support contained in $[0,T]\times (\bb T\backslash\{b_1,\ldots,b_k\})$.
  Hence, for every
$\varepsilon >0$:
\begin{equation}\label{limsup15}
\varlimsup_{\delta\to 0} \varlimsup_{N\to\infty}
\bb E_{\mu^N}^\beta \Big[ \max_{1\le i\le k} \Big\{
\int_0^T V_N(\varepsilon, H_i (s, \cdot) , \eta_s^{\delta N} ) \,
ds \Big\} \Big] \;\le\; K_0\; .
\end{equation}
\end{lemma}
\begin{proof}
It follows from Lemma \ref{replace2} that in order to prove
\eqref{limsup15}, we just need to show that
\begin{equation*}
\varlimsup_{N\to\infty} \bb E_{\mu^N}^\beta \Big[ \max_{1\le i\le k} \Big\{
\int_0^T V_N(\varepsilon,  H_i (s, \cdot) , \eta_s ) \, ds \Big\}
\Big] \;\le\; K_0\;.
\end{equation*}
By the entropy and the Jensen's inequality,
for each fixed $N$, the previous expectation is less than or equal to
\begin{equation*}
\frac {H(\mu^N \vert \nu_{\alpha})}{ N} \; +\; \frac 1{N}
\log \bb E_{\nu_{\alpha}} \Big[ \exp\Big\{
\max_{1\le i\le k}  N \int_0^T \,
V_N(\varepsilon,  H_i (s, \cdot) , \eta_s )  ds\Big\} \Big] \; .
\end{equation*}
By \eqref{f06}, the first term above is bounded by $K_0$.  Since $\exp\{
\max_{1\le j\le k} a_j \}$ is bounded from above by $\sum_{1\le j\le k}
\exp\{a_j\}$ and by \eqref{log bounds}, the limit as $N\uparrow\infty$, of the
second term of the previous expression is less than or equal to
\begin{equation*}
\max_{1\le i \le k} \varlimsup_{N\to\infty} \frac 1{N} \log
\bb E_{\nu_{\alpha}} \Big[ \exp
\Big\{ N  \int_0^T \,  V_N(\varepsilon,  H_i (s, \cdot) , \eta_s)ds
\Big\} \Big] \; .
\end{equation*}
We now prove that, for each fixed $i$ the limit above is nonpositive.

Fix $1\le i\le k$.
By the Feynman-Kac's formula and the variational formula for the
largest eigenvalue of a symmetric operator, for each fixed
$N$, the previous expectation is bounded from above by
\begin{equation*}
\int_0^T \, \sup_{f} \Big\{  \int V_N(\varepsilon,
H_i (s, \cdot) , \eta )
f(\eta) \nu_{\alpha} (d\eta) - N \mf D_N (f) \Big\}\; ds.
\end{equation*}
In last formula the supremum is taken over all probability densities
$f$ with respect to $\nu_{\alpha}$.  By assumption, each of the functions $\{H_i:i=1,\ldots,k\}$ vanishes in a neighborhood of each $b_i\in \bb T$.
This together with Lemma \ref{lemma55},
imply that the previous expression has nonpositive $\textrm{limsup}$. This is enough to conclude.
\end{proof}

We define now an approximation of the identity in the continuous torus given by
\begin{equation}\label{iota}
\iota_\eps(u,v)=\left\{\begin{array}{ll}
\pfrac{1}{\eps}\,\textbf 1_{(v,v+\eps)}(u)\,, &  \mbox{~if}\,\,\,\,v\in \bb T\backslash \displaystyle\cup_{i=1}^k(b_i-\eps, b_i)\,,\\
\quad\\
\pfrac{1}{\eps}\,\textbf 1_{(b_1-\eps,b_1)}(u)\,, &  \mbox{~if}\,\,\,\,v\in (b_1-\eps, b_1)\,,\\
\qquad\quad \vdots& \qquad\qquad \vdots\\
\pfrac{1}{\eps}\,\textbf 1_{(b_k-\eps,b_k)}(u)\,, &  \mbox{~if}\,\,\,\,v\in (b_k-\eps, b_k)\,.\\
\end{array}
\right.
\end{equation}\\
The convolution of a measure $\pi$ with $\iota_\eps$ is defined by
\begin{equation*}
 (\pi*\iota_\eps)(v)\;=\;\int\iota_\eps(u,v)\,\pi(du)\,.
\end{equation*}
 For a function $\rho$, the convolution $\rho*\iota_\eps$ is understood
as the convolution of the measure $\rho(u)\,du$ with $\iota_\eps$. Recall Definition \ref{def5}.
At this point, an important remark is the equality
\begin{equation}\label{conv}
\eta^{\eps N}_t(x)=(\pi^N_t*\iota_\eps)(\pfrac{x}{N})\,,
\end{equation}
which is of straightforward verification.

\begin{proof}[Proof of Lemma \ref{s03}]

 Consider a sequence $\{H_i:\, i\ge 1\}$ dense (with respect to the norm $\|H\|_{\infty}+\|\partial_uH\|_{\infty}$) in the subset of $C^{0,1}([0,T]\times \bb T)$
of functions with support contained in $[0,T]\times (\bb T\backslash\{b_1,\ldots,b_k\})$.

Recall that we suppose that $\{\bb Q_{\mu_N}^{\beta,N}:N\geq{1}\}$ converges
to $\bb Q_*^{\beta}$.
By \eqref{limsup15} and \eqref{conv}, for every $k\ge 1$,
\begin{equation*}
 \begin{split}
\varlimsup_{\delta\to 0} E_{\bb Q_*^{\beta}}\Big[ \max_{1\le i\le k}
\Big\{ \frac 1{\varepsilon} \int_0^T \, \int_{\bb T}  \,
H_i (s,u) \, &\Big\{ \rho^\delta_s (u) -
\rho^\delta_s (u + \varepsilon) \Big\}\hspace{0.1cm}du\hspace{0.1cm}ds \\
& -\; 2 \int_0^T\,
\int_{\bb T} \, (H_i(s,u))^2
\hspace{0.1cm}du\hspace{0.1cm}ds\Big\} \Big]\; \le \; K_0\; ,
 \end{split}
\end{equation*}
where $\rho^\delta_s (u) = (\rho_s * \iota_\delta)(u)$ as defined above.
Letting $\delta\downarrow 0$, performing a change of variables and then letting
$\varepsilon\downarrow 0$, we obtain that
\begin{eqnarray*}
\!\!\!\!\!\!\!\!\!\!\!\!\! &&
E_{\bb Q_*^{\beta}}\Big[ \max_{1\le i\le k} \Big\{
\int_0^T \, \int_{\bb T} (\partial_u H_i) (s,u)
\rho (s,u) \,\hspace{0.1cm}du\hspace{0.1cm}ds \\
\!\!\!\!\!\!\!\!\!\!\!\!\! && \qquad\qquad\qquad\qquad\qquad\qquad
- \;2 \int_0^T\, \int_{\bb T} (H_i(s,u))^2 \,\hspace{0.1cm}du\hspace{0.1cm}ds
\Big\} \Big] \;\le \; K_0\; .
\end{eqnarray*}
To conclude the proof it remains to apply the Monotone Convergence
Theorem and recall that $\{H_i: \, i\ge 1\}$ is a dense sequence (with respect to the norm $\|H\|_{\infty}+\|\partial_uH\|_{\infty}$)
in the subset of functions of $C^{0,1}([0, T]\times \bb T)$ with
support contained in $[0,T]\times(\bb T\backslash\{b_1\ldots,b_k\})$.
\end{proof}
\begin{remark}
In terms of Sobolev spaces, we have just proved that, for $\beta\in(1,\infty)$, $Q_*^\beta$-almost surely,
the limit trajectory $\rho(t,u)du$ is such that $\rho(t,u)$ belongs to
$L^1(0,T; \mc H^1(b_i,b_{i+1}))$, in each cylinder $[0,T]\times(b_i,b_{i+1})$. Notice that in view of the presence of slow bonds and
 of Lemma \ref{lemma55} is it not possible to obtain the same result considering the whole space
$L^1(0,T; \mc H^1(\bb T))$.
\end{remark}

\section{Characterization of Limit Points}\label{s6}

We prove in this section that all limit points $\bb Q_*^\beta$ of the
sequence $\{\bb Q^{\beta,N}_{\mu_N}: N\geq{1}\}$ are concentrated on trajectories of measures absolutely
continuous with respect to the Lebesgue measure: $\pi(t,du) = \rho(t,u) du$, whose density
$\rho(t,u)$ is a weak solution of the hydrodynamic equation
\eqref{edp1}, \eqref{edp2} or \eqref{edp3}, for each corresponding value of $\beta$.

Let $\bb Q_*^\beta$ be a limit point of the sequence $\{\bb Q^{\beta,N}_{\mu_N}: N\geq{1}\}$
and assume, without lost of generality, that $\{\bb Q^{\beta,N}_{\mu_N}: N\geq{1}\}$
converges to $\bb Q_*^\beta$. The existence of $\bb{Q}^\beta_*$ is guaranteed by Proposition \ref{s06}.

Since there is at most one particle per site, it is easy to show that $\bb
Q_*^\beta$ is concentrated on trajectories $\pi_t(du)$ which are absolutely
continuous with respect to the Lebesgue measure, $\pi_t(du) =
\rho(t,u) du$ and whose density $\rho(\cdot)t,\cdot$ is non-negative and bounded by
1 (for more details see \cite{kl}). We distinguish the regime of $\beta$ in different subsections below. In all the cases, we will make use of the martingale $M^{N}_t(H)$ defined in \eqref{M}. By a simple change of variables, the integral term in \eqref{M} can be rewritten as a function of the empirical measure, such that:
\begin{equation}\label{martingale2}
M^{N}_t(H)\;=\;\<\pi^N_t, H\> \,-\, \<\pi^N_0, H \> \,-\,
\int_0^t  \, \<\pi^N_s ,N^2\,\bb L_N H \> \,ds\,,
\end{equation}
 where $\bb{L}_N$ was defined in \eqref{generator random walk}.

 We notice here that, for any choice of $H$, $M^{N}_t(H)$ is a martingale. In due course we impose extra conditions on $H$ in order to identify the density $\rho(t,\cdot)$ as a weak solution of the corresponding weak equation depending on the regime of the parameter $\beta$.

\subsection{Characterization of Limit Points for $\beta\in [0,1)$}

\quad
\vspace{0.2cm}

Here, we want to show that $\rho(t,\cdot)$ is a weak solution of \eqref{edp1}.
Let $H\in C^2(\bb T)$. We begin by claiming  that
\begin{equation}\label{Q*1}
 \bb Q_{*}^\beta \Big[\,\pi_\cdot:\,
  \<\pi_t, H \> \,-\,
 \<\pi_0,H \> \,-\,
 \int_0^t  \,  \<\pi_s , \partial_u^2 H \> \,ds \,=\,0,\,\forall t\in[0,T]\, \Big]\;=\;1.
  \end{equation}
In order to prove the last claim, it is enough to show that, for every $\delta >0$:
\begin{equation*}
\bb Q_{*}^\beta \Big[\,\pi_{\cdot}: \, \sup_{0\le t\le T}
\Big\vert \<\pi_t, H \> \,-\,
\<\pi_0,H \> \,-\,
\int_0^t  \,  \<\pi_s , \partial_u^2 H \> \,ds\, \Big\vert
 \, > \, \delta\, \Big]\,=0.
 \end{equation*}
By Portmanteau's Theorem and Proposition \ref{A3}, last probability is bounded from above by
 \begin{equation*}
 \varliminf_{N\to\infty} \bb Q^{\beta,N}_{\mu_N} \Big[\,\pi_\cdot: \, \sup_{0\le t\le T}
\Big\vert \<\pi_t, H \> \,-\,
\<\pi_0,H \> \,-\,
\int_0^t  \,  \<\pi_s ,\partial_u^2 H \> \,ds \,\Big\vert
 \, > \, \delta\,\Big]\,
\end{equation*}
since the supremum above is a continuous function in the Skorohod metric.
Adding and subtracting $\<\pi^N_s,N^2\,\bb L_N H\>$ in the integral term above and recalling the definition of $\bb Q^{\beta,N}_{\mu_N}$,
the previous expression is bounded from above by
\begin{equation*}
\begin{split}
&\varlimsup_{N\to\infty} \bb P^\beta_{\mu_N} \Big[ \, \sup_{0\le t\le T}
\Big\vert \<\pi_t^N, H \> \,-\,
\<\pi_0^N,H \> \,-\,
\int_0^t  \,  \<\pi_s^N , N^2\,\bb L_N H \> \,ds\, \Big\vert
 \, > \, \delta/2\, \Big]\\
 +& \varlimsup_{N\to\infty} \bb P^\beta_{\mu_N} \Big[ \, \sup_{0\le t\le T} \Big\vert \int_0^t
 \<\pi_s^N,\partial_u^2 H-N^2\,\bb L_N H\> \,ds \,\Big\vert
 \, > \, \delta/2 \,\Big]\,.
\end{split}
\end{equation*}
By \eqref{martingale2} and \eqref{limprob}, the first term in last expression is null.
By the definition of $\Gamma_N$ given in Section \ref{s4} and since there is only one particle per site, the second term in last expression
 becomes bounded by
\begin{equation*}
\begin{split}
 &\varlimsup_{N\to\infty} \bb P^{\beta}_{\mu_N} \Big[ \,
\pfrac{T}{N}\sum_{x\notin \Gamma_N}\Big\vert \partial_u^2 H\Big(\frac{x}{N}\Big)- N^2\,\bb L_N H\Big(\frac{x}{N}\Big)\Big\vert
 \, > \, \delta/4 \,\Big]\\
 +&\,\varlimsup_{N\to\infty} \bb P^{\beta}_{\mu_N} \Big[ \,
 \sup_{0\le t\le T} \Big\vert\int_0^t  \pfrac{1}{N}\sum_{x\in \Gamma_N}
\Big\{\partial_u^2 H(\pfrac{x}{N})-N^2\,\bb L_N H(\pfrac{x}{N})\Big\}\,\eta_s(x)\,ds\, \Big\vert
  \, >\,\delta/4 \,\Big]\,.
\end{split}
\end{equation*}
Outside $\Gamma_N$, the operator $N^2\,\bb L_N$ coincides with the discrete Laplacian and since $H\in C^2(\bb T)$, the first term in last expression is zero.
Recall that there are $2k$ elements in $\Gamma_N$. Applying the triangular inequality, the second expression in the previous sum
 becomes bounded by
\begin{equation*}
\begin{split}
 &\varlimsup_{N\to\infty} \bb P^{\beta}_{\mu_N} \Big[ \,
\pfrac{2kT}{N}\|\partial_u^2H\|_{\infty}\,>\, \delta/8 \,\Big]\\
 +\,&\varlimsup_{N\to\infty} \bb P^{\beta}_{\mu_N} \Big[ \, \sup_{0\le t\le T}
\Big\vert\sum_{x\in \Gamma_N} \int_0^t
 N\,\bb L_N H(\pfrac{x}{N}) \,\eta_s(x)\,ds\,\Big\vert
 \, > \, \delta/8 \,\Big]\,.
\end{split}
\end{equation*}
For large $N$, the first probability vanishes.
Now we deal with the second term. We associate to each slow bond containing a point $b_i$, a unique pair of sites in $\Gamma_N$, namely $Nb_i$ and $Nb_i +1$.
By the triangular inequality, in order to show that the second expression above is zero, it is sufficient
to verify that
\begin{equation*}
\begin{split}
\varlimsup_{N\to\infty} \bb P^{\beta}_{\mu_N} \Big[ \,\sup_{0\le t\le T}\Big\vert
 \int_0^t   \{&N\,\bb L_N H(\pfrac{Nb_i}{N}) \,\eta_s(Nb_i)\\
 +&N\,\bb L_N H(\pfrac{Nb_i+1}{N}) \,\eta_s(Nb_i+1)\}\,ds\,\Big\vert
 \, > \, \delta/8k \,\Big]\,=0,
\end{split}
\end{equation*}
for each $i=1,\ldots,k$. The expression inside the integral above can be explicitly written  as
\begin{equation*}
\begin{split}
 &\Big\{N\,[H(\pfrac{Nb_i-1}{N})-H(\pfrac{Nb_i}{N})]
+N^{1-\beta}\,[H(\pfrac{Nb_i+1}{N})-H(\pfrac{Nb_i}{N})]\Big\}\,\eta_s(Nb_i)\\
+\,&\Big\{N^{1-\beta}\,[H(\pfrac{Nb_i}{N})-H(\pfrac{Nb_i+1}{N})]
+N\,[H(\pfrac{Nb_i+2}{N})-H(\pfrac{Nb_i+1}{N})]\Big\}\,\eta_s(Nb_i+1)\,.
\end{split}
\end{equation*}
Since $H$ is smooth and $\beta\in[0,1)$, the terms inside the parenthesis involving $N^{1-\beta}$ converge to zero
and the terms involving
$N$ converge to plus or minus the space derivative of $H$ at $b_i$. Therefore, again by the triangular inequality, it remains to show that, for any $\delta>0$,
\begin{equation}\label{expression2}
\varlimsup_{N\to\infty} \bb P^{\beta}_{\mu_N} \Big[ \,\sup_{0\le t\le T}\Big\vert
 \int_0^t
\partial_u H(b_i)\Big\{\eta_s(Nb_i)
 \,-\,\eta_s(Nb_i+1)\Big\}\,ds\,\Big\vert
 \, > \, \delta \,\Big]\,
\end{equation}
equals to zero.
The integral inside the probability above is continuous as a function of the time $t$. Moreover, it has a Lipschitz constant
 bounded by $\vert \partial_u H(b_i)\vert$.
If $\partial_u H(b_i)=0$, then there is nothing to do. Otherwise, let $t_0=0<t_1<\cdots<t_n=T$ be a partition of $[0,T]$ with mesh bounded by
$\delta (\vert 2 \partial_u H(b_i)\vert)^{-1}$. Notice the partition is fixed, depending only on the function $H$. By the triangular inequality, \eqref{expression2} is bounded by
\begin{equation*}
\sum_{j=0}^n\varlimsup_{N\to\infty} \bb P^{\beta}_{\mu_N} \Big[ \,\Big\vert
 \int_0^{t_j} \partial_u H(b_i)\Big\{\eta_s(Nb_i)  \,-\,\eta_s(Nb_i+1)\Big\}\,ds\,\Big\vert
 \, > \, \delta/2 \,\Big]\,.
\end{equation*}
Therefore, we just need to prove that, for any $\delta>0$ and any $t\in[0,T]$
\begin{equation*}
\varlimsup_{N\to\infty} \bb P^{\beta}_{\mu_N} \Big[ \,\Big\vert
 \int_0^{t}
\Big\{\eta_s(Nb_i)
 \,-\,\eta_s(Nb_i+1)\Big\}\,ds\,\Big\vert
 \, > \, \delta \,\Big]\,=0.
\end{equation*}
Applying Markov's inequality, we bound the previous probability by
\begin{equation*}
   \delta^{-1}\; \bb E^{\beta}_{\mu_N}   \Big[\,
\Big\vert \int_0^{t} \Big\{\eta_s(Nb_i)  \,-\,\eta_s(Nb_i+1)\Big\}\,ds\,\Big\vert\,
\Big]\,.
\end{equation*}
 Now, in order to conclude it is enough to do the following. First add and subtract the empirical mean in the box of size $\eps N$ around $Nb_i$ and $Nb_{i}+1$. Then, by the triangular inequality and since
 $|\eta_s^{\eps N}(x)-\eta_s^{\eps N}(x+1)|\leq \pfrac{2}{\eps N}$, the term involving the two empirical means vanish. For the other two terms, we invoke Lemma \ref{replace1}. This finishes the claim.
\begin{proposition} \label{p61} For $\beta\in[0,1)$, any limit point of $\bb Q^{\beta,N}_{\mu_N}$ is concentrated in
 absolutely continuous paths
$\pi_t(du) = \rho(t,u)\, du$, with positive density $\rho(t,\cdot)$ bounded by $1$, such that $\rho(t,\cdot)$ is a
weak solution of \eqref{edp1}.
\end{proposition}
\begin{proof}
Let $\{H_i: i\geq{1}\}$ be a countable dense set of functions on $C^2(\bb T)$, with respect to the norm $\|H\|_{\infty}+\|\partial_u^2H\|_\infty$. Provided by \eqref{Q*1} and intercepting a countable number of sets of probability one,
is straightforward to extend \eqref{Q*1} for
all functions $H\in C^2(\bb T)$ simultaneously.
\end{proof}

\subsection{Characterization of Limit Points for $\beta=1$} \label{limit points critical beta}

\quad
\vspace{0.2cm}

The idea in this case is to show that $\rho(t,\cdot)$ is an integral solution of \eqref{edp2} for
a small domain of functions and then extend this set to $\mc H^1_W$.\\

Let $\mc C_W\subset \mc H^1_W$ be the set of functions $H$ in $L^2(\bb T)$ such that for $x \in{\bb T}$
\begin{equation*}
H(x) \;=\; a \;+\; \int_{(0,x]} \Big(b+ \int_0^y h(z) dz\Big)  W(dy),
\end{equation*}
for some function $ h$ in $C(\bb T)$ and $a,b\in{\mathbb{R}}$ satisfying
\begin{equation*}
\int_0^1  h(x) \, dx \;=\; 0\;, \quad
\int_{(0,1]}  \Big( b + \int_0^y h(z) \, dz \Big) W(dy)\;=\;0\; .
\end{equation*}
Note that a function in $ \mc C_W$ is continuous in $ \bb T\backslash\{ b_1,...,b_k\}$ and well defined everywhere.
Now, fix a function $H\in \mc C_W$ and define the martingale
$M^{N}_t(H)$ as in \eqref{M}.
We aim that, for every $\delta>0$, the result in \eqref{limprob} holds for $H\in \mc C_W$. In fact, this
was already shown, for $H\in C^{2}(\bb T)$, in the proof of Proposition \ref{s06}.
By \eqref{varquad}, for $t\in[0,T]$
\begin{equation*}
\<M^{N}(H)\>_t\leq T\sum_{x\in\bb T_{N}} \xi^{N}_{x,x+1}
\Big[H(\pfrac{x+1}{N})-H(\pfrac{x}{N})\Big]^2.
\end{equation*}
Since $H\in\mc C_W$, $H$ is differentiable with
bounded derivative, except at the points $b_1,\ldots,b_k$. Therefore, for any pair $x,x+1$ such that there is no $b_i$ between $\pfrac{x}{N}$ and $\pfrac{x+1}{N}$, the following inequality holds
\begin{equation*}
 \xi^N_{x,x+1}\Big[\,H(\pfrac{x+1}{N})-H(\pfrac{x}{N})\,\Big]^2\leq \frac{1}{N^2}\|\partial_u^2H\|_\infty^2.
\end{equation*}
 On the other hand, if there is some $\{b_i: i=1,..,k\}$ in the interval $[\frac{x}{N},\pfrac{x+1}{N})$, then $\xi^N_{x,x+1}=N^{-\beta}$
and in this case we get to:
\begin{equation*}
 \xi^N_{x,x+1}\Big[\,H(\pfrac{x+1}{N})-H(\pfrac{x}{N})\,\Big]^2\leq \frac{4}{N^{2\beta}}\|H\|_{\infty}^2\,.
\end{equation*}
Since there are only finite $k$ slow bonds, we conclude that the quadratic variation of $M^N_t(H)$ vanishes as $N\rightarrow{\infty}$.
Now, Doob's inequality is enough to conclude.
As above, by a simple change of variables, we may rewrite the martingale $M^{N}_t(H)$ in terms of the
empirical measure as in \eqref{martingale2}. Now we want to analyze the integral term in the martingale decomposition \eqref{martingale2}.

\begin{lemma}\label{disc}
For any $H\in \mc C_W$,
\begin{equation*}
\lim_{N\to\infty}\pfrac{1}{N}\sum_{x\in \bb T_{N}}\Big\vert\, N^{2}\bb{L}_{N}H(\pfrac{x}{N})
-\pfrac{d}{dx}\pfrac{d}{dW}H(\pfrac{x}{N}) \,\Big\vert\; =\;0\,.
\end{equation*}
\end{lemma}
\begin{proof} Recall the definition of the set $\Gamma_N$ given in Section \ref{s4} and rewrite the previous sum as
\begin{equation}\label{eq1}
\pfrac{1}{N}\sum_{x\notin \Gamma_N}\Big\vert\,N^{2}\bb{L}_{N}H(\pfrac{x}{N})-\pfrac{d}{dx}\pfrac{d}{dW}H(\pfrac{x}{N})\Big\vert\,
+\pfrac{1}{N}\sum_{x\in \Gamma_N}\Big\vert\,N^{2}\bb{L}_{N}H(\pfrac{x}{N})-
\pfrac{d}{dx}\pfrac{d}{dW}H(\pfrac{x}{N})\Big\vert\,.
\end{equation}
Outside $b_1,\ldots,b_k$, the operator $\pfrac{d}{dx}\pfrac{d}{dW}$ coincides with the Laplacian,
and outside $\Gamma_N$, the discrete operator $N^2\,\bb L_N$ coincides with the discrete Laplacian. Hence, the first term above is equal to
\begin{equation*}
\pfrac{1}{N}\sum_{x\notin \Gamma_N}
\Big\vert \,N^{2}\Big(H(\pfrac{x+1}{N})+H(\pfrac{x-1}{N})-2H(\pfrac{x}{N})\Big)
-\partial_u^2 H(\pfrac{x}{N})\,\Big\vert\,.
\end{equation*}
It is easy to verify that $H\in C^2(\bb T\backslash \{b_1,\ldots,b_k\})$ and has bounded derivatives. Thus,
by a Taylor expansion on $H$, it follows that the previous sum converges to zero as $N\rightarrow{+\infty}$. On the other hand,
the second sum in \eqref{eq1} is bounded by the sum of
\begin{equation*}
\pfrac{1}{N}\sum_{x\in \Gamma_N}\Big\vert\, \pfrac{d}{dx}\pfrac{d}{dW}H(\pfrac{x}{N})\,\Big\vert
\end{equation*}
and
\begin{equation*}
\sum_{x\in \Gamma_N}
\Big\vert\, N\xi^N_{x,x+1}\Big[H(\pfrac{x+1}{N})-H(\pfrac{x}{N})\Big] + N\xi^N_{x-1,x}\Big[H(\pfrac{x-1}{N})-H(\pfrac{x}{N})\Big]\, \Big\vert \, .
\end{equation*}
Since $H\in C_W$, $\pfrac{d}{dx}\pfrac{d}{W}H$ is a continuous function, therefore bounded. Since $\Gamma_N$ has $k$ elements, the first sum above converges
to zero as $N\rightarrow{+\infty}$. It remains to analyze the second sum above, where now the definition of the domain $C_W$ is crucial.
For each $x\in\Gamma_N$, one of the conductances above is equal to $N^{-1}$. Let us suppose that $\xi^N_{x,x+1}=N^{-1}$ and
$\xi^N_{x-1,x}=1$, the other case being completely analogous. In this case,
there exists some $b_i\in(\pfrac{x}{N},\pfrac{x+1}{N}]$. From the definition of $C_W$ and the measure $W$,
the function $H$ has a discontinuity
at $b_i$ of size
\begin{equation*}
\int_0^{b_i}h(dz)\,dz\,.
\end{equation*}
Besides that, the function $H$ has also sided-derivatives at $b_i$ of the same value. With this in mind, is easy to see that
\begin{equation*}
[H(\pfrac{x+1}{N})-H(\pfrac{x}{N})] + N[H(\pfrac{x-1}{N})-H(\pfrac{x}{N})]
\end{equation*}
converges to zero as $N\to\infty$. Recalling there are finite $2k$ elements in $\Gamma_N$, we finish the proof of the lemma.
\end{proof}

Now, fix $H\in C_W$ and take a continuous function $H^\eps$ which coincides with $H$ in
$\bb T\backslash \cup_{i=1}^k(b_i-\eps, b_i+\eps )$ and that
$\|H^\eps\|_{\infty} \leq \|H\|_\infty$.
 The choice of $\eps$ will be determined later.  Notice that
\begin{equation*}
\sup_{0\le t\le T}|\<\pi_t,H^\eps-H\>|\leq \sup_{0\le t\le T}\sum_{i=1}^{k}
\int_{(b_i-\eps,b_i+\eps)}\!\!\!\!\!\!\!\!\!\!\!\rho(t,u)\,
|H^\eps(u)-H(u)|\,du\leq 4\,k\,\eps\,\|H\|_\infty\,.
\end{equation*}
For every $\delta >0$,
\begin{equation}\label{Q*}
\bb Q_{*}^\beta \Big[\,\pi_\cdot: \, \sup_{0\le t\le T}
\Big\vert \<\pi_t, H \> \,-\,
\<\pi_0,H \> \,-\,
\int_0^t  \,  \<\pi_s ,\pfrac{d}{dx}\pfrac{d}{dW} H \> \,ds \Big\vert
 \, > \, \delta\, \Big]
 \end{equation}
\begin{equation*}
\begin{split}
\leq&\;\bb Q_{*}^\beta \Big[\,\pi_\cdot: \, \sup_{0\le t\le T}
\Big\vert\, \<\pi_t, H^\eps \> \;-\;
\<\pi_0,H^\eps \> \;-\;
\int_0^t  \,  \<\pi_s ,\pfrac{d}{dx}\pfrac{d}{dW} H \> \,ds \Big\vert
 \, > \, \delta/3 \,\Big]\\
 +&2\,\bb Q_{*}^\beta \Big[ \,\pi_\cdot:\, \sup_{0\le t\le T}
\Big\vert \<\pi_t,H^\eps-H\> \Big\vert
 \, > \, \delta/3\, \Big]\,.
\end{split}
\end{equation*}
 By a suitable choice of $\eps$, the second probability in the sum above is null.
 Since  $H^\eps$ and $\pfrac{d}{dx}\pfrac{d}{dW}H$ are continuous,
by the Portmanteau's Theorem and Proposition \ref{A3}, it holds that
\begin{equation*}
\begin{split}
&\,\bb Q_{*}^\beta \Big[\,\pi: \, \sup_{0\le t\le T}
\Big\vert \<\pi_t, H^\eps \> \;-\;
\<\pi_0,H^\eps \> \;-\;
\int_0^t  \,  \<\pi_s ,\pfrac{d}{dx}\pfrac{d}{dW} H \> \,ds \Big\vert
 \, > \, \delta/3 \,\Big]\\
\leq\,&\varliminf_{N\to\infty} \bb Q^{\beta,N}_{\mu_N} \Big[\,\pi: \, \sup_{0\le t\le T}
\Big\vert \<\pi_t, H^\eps \> \;-\;
\<\pi_0,H^\eps \> \;-\;
\int_0^t  \,  \<\pi_s ,\pfrac{d}{dx}\pfrac{d}{dW} H \> \,ds \Big\vert
 \, > \, \delta/3\,\Big]\\
 =\,&\varliminf_{N\to\infty} \bb P^{\beta}_{\mu_N} \Big[ \, \sup_{0\le t\le T}
\Big\vert \<\pi_t^N, H^\eps \> \;-\;
\<\pi_0^N,H^\eps \> \;-\;
\int_0^t  \,  \<\pi_s^N ,\pfrac{d}{dx}\pfrac{d}{dW} H \> \,ds \Big\vert
 \, > \, \delta/3\,\Big]\,.
\end{split}
\end{equation*}
Notice that the last equality is just the definition of the measure $\bb{Q}^{\beta,N}_{\mu_N}$.
Since there is only one particle per site, it holds that $\sup_{0\le t\le T} \big\vert \<\pi_t^N,H^\eps-H\> \big\vert\leq 4\,k\,\eps \|H\|_\infty\,,$
since $H^\eps$ coincides with $H$ in $\bb T\backslash \cup_{i=1}^k(b_i-\eps, b_i+\eps )$.
Adding and subtracting $\<\pi^N_s,N^2\,\bb L_N H\>$,  $\<\pi_t^N,H\>$ and $\<\pi_0^N,H\>$, we obtain that
\begin{equation*}
\begin{split}
 &\varliminf_{N\to\infty} \bb P^{\beta}_{\mu_N} \Big[ \, \sup_{0\le t\le T}
\vert \<\pi_t^N, H^\eps \> \;-\;
\<\pi_0^N,H^\eps \> \;-\;
\int_0^t  \,  \<\pi_s^N ,\pfrac{d}{dx}\pfrac{d}{dW} H \> \,ds \vert
 \, > \, \delta/3\,\Big]\\
 \leq&\varlimsup_{N\to\infty} \bb P^{\beta}_{\mu_N} \Big[ \, \sup_{0\le t\le T}
\Big\vert \<\pi_t^N, H \> \,-\,
\<\pi_0^N,H \> \,-\,
\int_0^t  \,  \<\pi_s^N ,N^2\,\bb L_N H \> \,ds \Big\vert
 \, > \, \delta/12\,\Big]\\
+\,&\varlimsup_{N\to\infty} \bb P^{\beta}_{\mu_N} \Big[ \,
\pfrac{1}{N}\sum_{x\in \bb T_{N}}\Big\vert N^2\bb{L}_{N}H(\pfrac{x}{N})
-\pfrac{d}{dx}\pfrac{d}{dW}H(\pfrac{x}{N}) \Big\vert
 \, > \, \delta/12 \,\Big]\\
+\,& 2\varlimsup_{N\to\infty} \bb P^{\beta}_{\mu_N} \Big[ \, \sup_{0\le t\le T}
\Big\vert \<\pi_t^N,H^\eps-H\> \Big\vert
 \, > \, \delta/12\, \Big]\,.
\end{split}
\end{equation*}
With another suitable choice of $\eps$, the third probability in the sum above is null.
Lemma \ref{disc} implies that the second probability above is zero for $N$ sufficiently large.
Recall we proved that \eqref{limprob} holds for $H\in {\mc C_W}$, so that the first term in the sum above is zero. Finally, from the previous computations we conclude that
\eqref{Q*} is zero for any $\delta>0$. Therefore,
$\bb Q_{*}^\beta$ is concentrated on absolutely continuous paths
$\pi_t(du) = \rho(t,u)\, du$ with positive density bounded by $1$ and for any fixed $H\in\mc C_W$,
$\bb Q_{*}^\beta$ a.s.
\begin{equation}\label{equation}
\<\rho_t,  H \> - \<\rho_0, H \> \;=\; \int_0^t  \, \Big\< \rho_s \,,\, \pfrac{d}{dx}\pfrac{d}{dW}  H \Big\> \,ds\,,\qquad
\textrm{for all}\;t\in[0,T]\,.
\end{equation}
\begin{proposition} For $\beta=1$, any limit point of $\bb Q^{\beta,N}_{\mu_N}$ is concentrated in absolutely continuous paths
$\pi_t(du) = \rho(t,u)\, du$, with positive density $\rho(t,\cdot)$ bounded by $1$, such that $\rho(t,\cdot)$ is
a weak solution of \eqref{edp2}.
\end{proposition}
\begin{proof}
By a density argument, \eqref{equation} also holds, $Q_*^\beta$ a.s., for all $H\in C_W$ simultaneously.
It remains to extend \eqref{equation} for $H\in \mc H^1_W$. For that purpose fix $H\in \mc H^1_W$. Thus, for $x\in{\bb T}$
\begin{equation*}
H(x)=\alpha + \int_{(0,x]}  \left(\beta+\int_0^y h(z) \, dz\right)W(dy) \,,
\end{equation*}
with $\alpha,\beta\in\bb R$, $h\in L^2(\bb T)$ satisfying \eqref{domain}. Let $h_n\in C(\bb T)$
 converging to $h\in L^2(\bb T)$. Define
\begin{equation*}
H_n(x)=\alpha_n + \int_{(0,x]}  \left(\beta_n+\int_0^y h_n(z) \, dz\right)W(dy)\,,\end{equation*}
where $\alpha_n\to \alpha$ and $\beta_n\to \beta$. By the Dominated Convergence Theorem, it follows
 that $H_n$ converges uniformly to $H$. Therefore
\eqref{equation} is true for all $H\in \mc H^1_W$.
\end{proof}

\subsection{Characterization of Limit Points for $\beta\in (1,\infty)$}

\quad
\vspace{0.2cm}

In this regime of the parameter $\beta$,
Proposition \ref{s05} says that $Q_*^\beta$ is concentrated on trajectories absolutely continuous with respect to the Lebesgue measure $\pi_t(du)=\rho(t,u)\,du$ such that,
for each interval $(b_i,b_{i+1})$,
$\rho(t,\cdot)$ belongs to $L^2(0,T;\mc H^1(b_i,b_{i+1}))$.
It is well known that the Sobolev space $\mc H^1(a,b)$ has the following properties: all its elements are
 absolutely continuous functions  with bounded variation, c.f. \cite{e} and \cite{l}, therefore with lateral limits well-defined. Such property is inherited by $L^2(0,T;\mc H^1(b_i,b_{i+1}))$ in the sense that we can
integrate in time the lateral limits.
Therefore, $Q_*^\beta a.s.$, for each $i=1,\ldots,k$ and for any $t\in [0,T]$:
\begin{equation*}
\int_0^t\rho(s,b_i^+)\,ds<\infty \quad \textrm{and}\quad \int_0^t\rho(s,b_{i+1}^-)\,ds\,<\infty.
\end{equation*}

To simplify notation, in this subsection we denote $a=b_i$ and $b=b_{i+1}$.
Fix $h\in C^2(\bb T)$ and define $H:[0,T]\times \bb T\to\bb R$ by
$H(t,u)\;=\;h(t,u)\,\textbf{1}_{[a,b]}(u)$.

Recall that $\pi_t(du)=\rho(t,u)du$.
We begin by claiming  that
\begin{equation}\label{Q*3}
\begin{split}
&\bb Q_{*}^\beta \Big[\,\pi_\cdot:\,\<\rho_t, H_t \> \;-\;\<\rho_0,H_0 \>\;-\; \int_0^t \,\<\rho_s ,\partial_u^2 H_s+\partial_s H_s \>\,ds\\
&-\int_0^t\partial_uH(s,a^+)\,\rho(s,a^+)\,ds
+\int_0^t\partial_u H(s,b^-)\,\rho(s,b^-)\,ds=0,\forall t\in[0,T]\,\Big]\;=\;1\,.
\end{split}
\end{equation}
In order to prove \eqref{Q*3}, it is enough to show that, for every $\delta >0$
\begin{equation*}
\begin{split}
\bb Q_{*}^\beta &\Big[\pi:\,\sup_{0\leq t\leq T}\Big\vert  \<\rho_t, H_t \> \;-\;
 \<\rho_0,H_0 \> \;-\; \int_0^t  \,  \<\rho_s , \partial_u^2 H_s+\partial_s H_s \> \,ds \\
&- \int_0^t\partial_uH(s,a^+) \,\rho(s,a^+)\,ds
+ \int_0^t \partial_u H(s,b^-)\,\rho(s,b^-)\,ds\Big\vert\;>\;\delta\, \Big]\,=0\,.
\end{split}
\end{equation*}
Since the boundary integrals are not well-defined in the whole Skorohod space $D([0,T],\mc M)$, we cannot use
directly Portmanteau's Theorem. To avoid this technical obstacle, fix $\eps>0$, which will be taken small later.
Adding and subtracting the convolution of $\rho(t,u)$ with $\iota_\eps$, the probability above is less than or equal to the sum of
\begin{equation}\label{e35}
\begin{split}
&\bb Q_{*}^\beta \Big[\,\pi_\cdot:\,\sup_{0\leq t\leq T}\Big\vert\<\rho_t, H_t \>\;-\;
\<\rho_0,H_0 \>\;-\;\int_0^t \,\<\rho_s ,\partial_u^2 H_s+\partial_s H_s \>\,ds\\
&-\int_0^t\partial_uH(s,a^+)\,(\rho_s*\iota_\eps)(a)\,ds
+\int_0^t\partial_u H(s,b^-)\,(\rho_s*\iota_\eps)(b-\eps)\,ds\Big\vert>\delta/2\,\Big]
\end{split}
\end{equation}
and
\begin{equation*}
\begin{split}
&\bb Q_{*}^\beta \Big[\,\pi:\,\sup_{0\leq t\leq T}\Big\vert
\int_0^t\partial_uH(s,a^+)\,(\rho_s*\iota_\eps)(a)\,ds-\int_0^t\partial_uH(s,b^-)\,(\rho_s*\iota_\eps)(b-\eps)\,ds \\
&-\int_0^t\partial_uH(s,a^+) \,\rho(s,a^+)\,ds
+ \int_0^t \partial_uH(s,b^-)\,\rho(s,b^-)\,ds\Big\vert\;>\;\delta/2\,\Big]\,.
\end{split}
\end{equation*}
where $\iota_\eps$ and the convolution $\rho*\iota_\eps$ were defined in \eqref{iota}. The convolutions above are
suitable  averages of $\rho$ around the boundary points $a$ and $b$. Therefore, as $\eps\downarrow 0$, the set inside
the previous probability decreases to a set of null probability. It remains to deal with \eqref{e35}.

By Portmanteau's Theorem, Proposition \ref{A3} and since there is only one particle per site, \eqref{e35} is bounded from above by
\begin{equation*}
\begin{split}
&\varliminf_{N\to\infty}\; \bb Q^{\beta,N}_{\mu_N}  \Big[\,\pi_\cdot:\,\sup_{0\leq t\leq T}\Big\vert\, \<\pi_t, H \>\,-\,
 \<\pi_0,H_0 \>\,-\,\int_0^t \,\<\pi_s ,\partial_u^2 H_s+\partial_s H_s\>\,ds\\
&-\int_0^t \partial_u H(s,a^+)\,(\pi_s*\iota_\eps)(a)\,ds
+ \int_0^t\partial_u H(s,b^-)\,(\pi_s*\iota_\eps)(b-\eps)\,ds\Big\vert\;>\;\delta/2\,\Big]\,.
\end{split}
\end{equation*}
Now, by the definition of  $\bb Q^{\beta,N}_{\mu_N}$, we can rewrite the previous expression as
\begin{equation*}
\begin{split}
&\varliminf_{N\to\infty}\; \bb P^{\beta}_{\mu_N}  \Big[\, \sup_{0\leq t\leq T}\Big\vert\,  \<\pi_t^N, H_t \> \,-\,
 \<\pi_0^N,H_0 \> \,-\, \int_0^t  \,  \<\pi_s^N , \partial_u^2 H_s+\partial_s H_s \> \,ds \\
&- \int_0^t \partial_u H(s,a^+)\,\eta^{\eps N}_s(Na+1)\,ds
+ \int_0^t \partial_uH(s,b^-)\,\eta^{\eps N}_s(Nb)\,ds\Big\vert\;>\;\delta/2\, \Big]\,.
\end{split}
\end{equation*}
If we consider the discrete torus as embedded in the continuous torus, $Na+1$ is the closest  site to the right of $a$ and $Nb$ is the closest site to the left of $b$.
The next step is to add and subtract $\<\pi_s^N,N^2\,\bb L_N H\>$
and the previous probability becomes now bounded from above by the sum of
\begin{equation*}
\varlimsup_{N\to\infty}\;\bb P^{\beta}_{\mu_N}\Big[\,\sup_{0\leq t\leq T}\Big\vert\,\<\pi_t^N, H_t \>\,-\,
 \<\pi_0^N,H_0 \>-\int_0^t\,\<\pi_s^N,N^2\,\bb L_N H_s+\partial_s H_s\>\,ds\,\Big\vert\,>\,\delta/4\,\Big]
\end{equation*}
and
\begin{equation*}
\begin{split}
\varlimsup_{N\to\infty}\; &\bb P^{\beta}_{\mu_N}\Big[\,\sup_{0\leq t\leq T}
\Big\vert \int_0^t\,\<\pi_s^N , N^2\,\bb L_N H_s \>\,ds\,-\,\int_0^t\<\pi_s^N,\partial_u^2 H_s\>\,ds \\
&-\int_0^t \partial_uH(s,a^+)\,\eta^{\eps N}_s(Na+1)\,ds
+\int_0^t\partial_uH(s,b^-)\,\eta^{\eps N}_s(Nb)\,ds\Big\vert\;>\;\delta/4\,\Big]\,.
\end{split}
\end{equation*}
Repeating similar computations to the ones performed in Section \ref{s4} we can show \eqref{limprob} for a test function $H$ that depends also on time. Therefore the first probability above is null. Now we focus on showing that the second probability above is null. Recalling the definition of $H(s,\cdot)$ above, we have that $H(s,\cdot)$ is zero outside the interval $[a,b]$.
Besides that, for the set of vertices $\{Na+2,\ldots,Nb-1\}$, the discrete operator $N^2\,\bb L_N$
coincides with the discrete Laplacian, which applied to $H(s,\cdot)$ converges uniformly to the continuous Laplacian of $H(s,\cdot)$. Hence,
by the triangular
inequality, it is enough to show that, for any $\delta>0$:
\begin{equation*}
\begin{split}
\varlimsup_{N\to\infty}\; &\bb P^{\beta}_{\mu_N}  \Big[\,\sup_{0\leq t\leq T}
\Big\vert  \pfrac{1}{N}\int_0^t  \,   \{N^2\,\bb L_N H_s(\pfrac{Na}{N}) -\partial_u^2 H_s(\pfrac{Na}{N})\}\,\eta_s(Na)  \,ds \\
&+\pfrac{1}{N}\int_0^t  \,   \{N^2\,\bb L_N H_s(\pfrac{Na+1}{N}) -\partial_u^2 H_s(\pfrac{Na+1}{N})\}\,\eta_s(Na+1)  \,ds \\
&+\pfrac{1}{N}\int_0^t  \,   \{N^2\,\bb L_N H_s(\pfrac{Nb}{N}) -\partial_u^2 H_s(\pfrac{Nb}{N})\}\,\eta_s(Nb)  \,ds \\
&+\pfrac{1}{N}\int_0^t  \,   \{N^2\,\bb L_N H_s(\pfrac{Nb+1}{N}) -\partial_u^2 H_s(\pfrac{Nb+1}{N})\}\,\eta_s(Nb+1)  \,ds \\
&-\int_0^t \partial_uH(s,a^+)\,\eta^{\eps N}_s(Na+1)\,ds
+\int_0^t\partial_uH(s,b^-)\,\eta^{\eps N}_s(Nb)\,ds\Big\vert\;>\delta\,\Big]=0.
\end{split}
\end{equation*}
Since $h\in C^2(\bb T)$, the term involving the Laplacian above is bounded. Now, by the triangular inequality,
it is sufficient to show that, for any $\delta>0$:
\begin{equation*}
\begin{split}
\varlimsup_{N\to\infty}&\bb P^{\beta}_{\mu_N}\Big[\sup_{0\leq t\leq T}
\Big\vert\int_0^tN \bb L_N H_s(\pfrac{Na}{N})\,\eta_s(Na)ds
+\int_0^tN\bb L_N H_s(\pfrac{Na+1}{N})\,\eta_s(Na+1)ds\\
&+\int_0^t\,N\,\bb L_N H_s(\pfrac{Nb}{N})\,\eta_s(Nb)\,ds+
\int_0^t\,N\,\bb L_N H_s(\pfrac{Nb+1}{N}) \,\eta_s(Nb+1)\,ds\\
&-\int_0^t\partial_u H(s,a^+)\, \eta^{\eps N}_s(Na+1)\,ds
+\int_0^t\partial_u H(s,b^-)\,\eta^{\eps N}_s(Nb)\,ds\,\Big\vert\;>\;\delta\,\Big]\,=0.
\end{split}
\end{equation*}
For each one of the four vertices appearing inside the previous probability, the
operator $\bb L_N$ has two conductances, one equals to $N^{-\beta}$ and the other equals to $1$.
Since $\beta>1$, the terms involving $N^{-\beta}$ converge to zero. The terms involving the conductances equal to $1$, converge
to plus or minus the lateral space derivatives of $H$. Recall from definition of $H$ that  $\partial_u H(s,a^-)=\partial_u H(s,b^+)=0$ for all $0\leq{s}\leq{t}$.
From this, it remains to show that for any $\delta>0$
\begin{equation*}
\begin{split}
\varlimsup_{N\to\infty}\;& \bb P^{\beta}_{\mu_N}  \Big[\,\sup_{0\leq t\leq T}
\Big\vert\int_0^t\partial_uH(s,a^+)\,\eta_s(Na+1)\,ds-\int_0^t\,\partial_uH(s,b^-)\,\eta_s(Nb)\,ds\\
&-\int_0^t\partial_uH(s,a^+) \,\eta^{\eps N}_s(Na+1)\,ds
+ \int_0^t \partial_uH(s,b^-)\,\eta^{\eps N}_s(Nb)\,ds\,\Big\vert\;>\;\delta\, \Big]\,,
\end{split}
\end{equation*}
is null. Last expression is bounded from above by
\begin{equation*}
\begin{split}
&\varlimsup_{N\to\infty}\; \bb P^{\beta}_{\mu_N}\Big[\,\sup_{0\leq t\leq T}
\Big\vert\int_0^t\partial_uH(s,a^+)\,\Big\{\eta_s(Na+1)-\eta^{\eps N}_s(Na+1)\Big\}\,ds\Big\vert\;>\;\delta/2\,\Big]\\
+\,&\varlimsup_{N\to\infty}\;\bb P^{\beta}_{\mu_N}\Big[\, \sup_{0\leq t\leq T}
\Big\vert \int_0^t \,\partial_uH(s,b^-)\,\Big\{\eta_s(Nb)-\eta^{\eps N}_s(Nb)\Big\}\,ds\Big\vert\;>\;\delta/2\,\Big]\,.
\end{split}
\end{equation*}
The integral inside the probability above is a continuous function of the time $t$. Moreover, it has a bounded
Lipschitz constant. The same argument as the one used in \eqref{expression2} together with
Lemma \ref{replace2} imply that the previous expression converges to zero when $\eps\downarrow 0$,  which proves \eqref{Q*3}.
 \begin{proposition} For $\beta\in(1,\infty)$, any limit point of $\{\bb Q^{\beta,N}_{\mu_N}:N\geq{1}\}$ is concentrated in  absolutely continuous paths
$\pi_t(du) = \rho(t,u)\, du$, with positive density $\rho(t,\cdot)$ bounded by $1$, such that $\rho(t,\cdot)$ is
a weak solution of \eqref{edp3} in each cylinder $[0,T]\times[b_i,b_{i+1}]$.
\end{proposition}
\begin{proof}
Given \eqref{Q*3}, it remains to extend the result for all functions $H$ and all cylinders $[0,T]\times[b_i,b_{i+1}]$
simultaneously. Intercepting a countable number of sets of probability one
 and applying a density argument as in Proposition \ref{p61}, the statement follows.
\end{proof}

\section{Uniqueness of Weak Solutions}\label{s7}
The uniqueness of weak solutions of \eqref{edp1} is standard and we refer to \cite{kl} for a proof.
It remains to prove uniqueness of weak solutions of the parabolic differential equations \eqref{edp2} and \eqref{edp3}. In both cases, by linearity it suffices to check the uniqueness for $\gamma(\cdot)\equiv 0$. Notice that existence of weak solutions of \eqref{edp1}, \eqref{edp2} and \eqref{edp3} is guaranteed by tightness
of the process as proved in Section \ref{s4}, together with the characterization of limit points as proved in Section \ref{s6}.

\subsection{Uniqueness of weak solutions of \eqref{edp2}}
\quad
\vspace{0.2cm}

Let  $\rho :\bb R_+ \times \bb T \to \bb R$ be a weak solution of  \eqref{edp2} with $\gamma\equiv 0$.
 By Definition \ref{def weak solution edp2}, for all $H\in{\mc H^1_W}$ and all $t>0$
\begin{equation}\label{eqint}
\< \rho_t, H\>
\;=\; \int_0^t \Big\< \rho_s , \frac{d}{dx}\frac{d}{dW}  H \Big\> \, ds\;.
\end{equation}
 From Theorem 1 of \cite{fl},
the operator  $- \frac{d}{dx}\frac{d}{dW}$  has a countable number of eigenvalues  $\{\lambda_n : n\ge 0\}$ and eigenvectors $\{F_n : n\ge 0\}$.
All eigenvalues have finite multiplicity,
  $0= \lambda_0 \le \lambda_1 \le \cdots$ and $\lim_{n\to\infty} \lambda_n  = \infty$.
Moreover, the eigenvectors $\{F_n: n\geq{0}\}$ form a complete orthonormal system in $L^2(\bb T)$. For $t>0$, define
\[R(t)=\sum_{n\in\bb{N}}\frac{1}{n^{2}(1+\lambda_n)}\<\rho_t,F_n\>^2.\]

Notice that $R(0)=0$
and since $\rho_t$ belongs to $L^2(\bb T)$, $R(t)$ is well defined for all $t\geq 0$.
By \eqref{eqint}, it follows that $\frac{d}{dt}\<\rho_t,F_n\>^2=
-2\lambda_n\<\rho_t,F_n\>^2$. Thus
\begin{equation*}
(\pfrac{d}{dt}R)(t)=-\sum_{n\in\bb N}\frac{2\lambda_n}{n^{2}(1+\lambda_n)}\< \rho_t,F_n\> ^{2}\,,
\end{equation*}
because
$\sum_{n\leq N}\frac{-2\lambda_n}{n^{2}(1+\lambda_n)}\<\rho_t,F_n\>^{2}$ converges uniformly to
$\sum_{n\in \bb N}\frac{-2\lambda_n}{n^{2}(1+\lambda_n)}\<\rho_t,F_n\>^{2}$,
as $N$ increases to infinity. Therefore $R(t)\geq0$ and $(\frac{d}{dt}R)(t)\leq 0$, for all $t>0$ and since $R(0)=0$, it follows that $R(t)=0$ for all $t>0$.
As a consequence of $\{F_n: n\geq{0}\}$ being a complete orthonormal system, it follows that $\<\rho_t,\rho_t\>=0$, which is enough to conclude.

\subsection{Uniqueness of weak solutions of \eqref{edp3}}
\quad
\vspace{0.2cm}

At first, we begin with an auxiliary lemma on integration by parts.
\begin{lemma}\label{l71}
 Let $\rho(t,\cdot)$ be a function in the Sobolev space $L^2(0,T;\mc H^1(a,b))$. Then, for any $H\in C^{0,1}([0,T]\times [a,b])$:
\begin{equation*}
\begin{split}
&\int_0^T\int_a^b \rho(s,u)\,\p_u H(s,u)\,du\,ds \\
=&-\int_0^T\int_a^b\p_u\rho(s,u)\,H(u,s)\,du\,ds+\int_0^T\Big\{\rho(s,b)\,H(s,b)-
\rho(s,a)\,H(s,a)\Big\}\,ds\,.
\end{split}
\end{equation*}
\end{lemma}

Notice the partial derivative in $\rho$ is the weak derivative, while the partial derivative in $H$ is the usual one.
Besides that, the function $H$ is smooth, but possibly not null at the boundary $[0,T]\times\{a,b\}$, and therefore is not
valid the integration by parts in the sense of $L^2(0,T;\mc H^1(a,b))$, which has no boundary integrals.

\begin{proof} Fix $\eps>0$ and write $H=H^\eps+(H-H^\eps)$, where $H^\eps$ coincides with $H$ in the region $[0,T]\times (a+\eps,b-\eps)$, has
compact support contained in $[0,T]\times (a,b)$ and belongs to $C^{0,1}([0,T]\times (a,b))$. By the assumptions
on $H^\eps$, we  have that
\begin{equation*}
\begin{split}
&\int_0^T\int_a^b\rho(s,u)\,\p_u H(s,u)\,du\,ds\\
=&-\int_0^T\int_a^b\p_u\rho(s,u)\,H^\eps(s,u)\,du\,ds  + \int_0^T\int_a^b\rho(s,u)\p_u(H-H^\eps)(s,u)\,du\,ds\,.\\
\end{split}
\end{equation*}
Last result is a consequence of $H^\eps$ having compact support strictly contained in the open set $(a,b)$.
Let  $f_\eps:[a,b]\to\bb R$  be the function such that $f(u)=1$ if $u\in(a+\eps,b-\eps)$, $f(a)=f(b)=0$,
and interpolated linearly otherwise. The decomposition $H=H\,f^\eps+ H(1-f^\eps)$ can be done, but now the function $H\,f^\eps$ does not have the properties as required above for $H^\eps$. Nevertheless, taking a suitable approximating sequence of functions $H^ \eps$, it follows that
\begin{equation*}
\begin{split}
&\int_0^T\int_a^b\rho(s,u)\,\p_u H(s,u)\,du\,ds\\
=&-\int_0^T\int_a^b\Big\{\p_u\rho(s,u)H(s,u)f^\eps(u)
 +\rho(s,u)\p_u\Big(H(s,u)(1-f^\eps(u))\Big)\Big\}\,du\,ds.
\end{split}
\end{equation*}
Taking the limit as $\eps\downarrow 0$ yields the statement of the lemma.
\end{proof}

Let $\rho(t,\cdot)$ be a weak solution of \eqref{edp3} with $\gamma\equiv{0}$. Provided by Lemma \ref{l71}, for any function
$H\in C^{1,2}([0,T]\times (b_i,b_{i+1}))$,
\begin{equation*}
\int_{b_i}^{b_{i+1}}\rho_t(u)H(t,u)\,du+\int_0^t
\int_{b_i}^{b_{i+1}}\Big\{\p_u\rho_s(u)\p_u H(s,u)-\rho_s(u)\p_s H(s,u)\Big\}du\,ds=0.
\end{equation*}
From this point, uniqueness is a particular case of a general result in \cite{la}, namely Theorem III.4.1.
In sake of completeness, we sketch an adaptation of it to our particular case. Denote by $W^{1}_{2,T}=W^{1}_{2,T}([0,T]\times (a,b))$ the space of functions with one weak derivative in space and time, both belonging to $L^2([0,T]\times (a,b))$ and vanishing at time $T$. By extending the previous equality to $H\in W^{1}_{2,T}$ it follows that
\begin{equation}\label{efinal}
\int_0^T\int_{b_i}^{b_{i+1}}\Big\{\p_u\rho_s(u) \, \p_u H(s,u)-\rho_s(u)\,\p_s H(s,u)\Big\} \,du\, ds\,=\,0\,.
\end{equation}
It is not difficult to show that the function
\begin{equation*}
H(s,u)\;=\; -\int_s^T\rho(r,u)\,dr
\end{equation*}
belongs to $W^{1}_{2,T}$. Replacing last function in \eqref{efinal}, then we can rewrite \eqref{efinal} as
\begin{equation*}
\int_0^T\int_{b_i}^{b_{i+1}}\Big\{\frac{1}{2}\p_s( \p_u H(s,u))^2-(\p_s H(s,u))^2\Big\}\,du\, ds\;=\;0\;.
\end{equation*}
By Fubini's Theorem we get to
\begin{equation*}
\frac{1}{2}\int_{b_i}^{b_{i+1}}\Big\{(\p_u H(T,u))^2\,-( \p_u H(0,u))^2\Big\}\,du
-\int_0^T\int_{b_i}^{b_{i+1}} (\p_s H(s,u))^2\,du\, ds\;=\;0\,.
\end{equation*}
By the definition of $H$, its weak space derivative vanishes at time $T$, so that the first integral above is null. Therefore, $\partial_sH$ is identically null, and by the definition of $H$ above, this implies that
$\rho$ vanishes, finishing the proof.

\section{Appendix}
\begin{proposition}\label{K0}
 Denote by $H_N (\mu_N | \nu_\alpha)$ the entropy of a probability
measure $\mu_N$ with respect to a stationary state $\nu_\alpha$. Then, there
exists a finite constant $K_0:=K_{0}(\alpha)$ such that
$H_N (\mu_N | \nu_\alpha) \;\le\; K_0 N\,,$
for all probability measures $\mu_N$.
\end{proposition}
\begin{proof}
 Recall that $\nu_\alpha$ is Bernoulli product of parameter $\alpha$. By the
explicit formula given in Theorem A1.8.3 of \cite{kl},
\begin{eqnarray*}
 H_N (\mu_N | \nu_\alpha) & = & \sum_{\eta\in \{0,1\}^{\bb T_N}}\mu_N (\eta)\,
\log \frac{\mu_N (\eta)}{\nu_\alpha (\eta)}\\
 & \leq & \sum_{\eta\in \{0,1\}^{\bb T_N}}\mu_N (\eta)\,
\log \frac{1}{\nu_\alpha(d\eta)}\\
& \leq & \sum_{\eta\in \{0,1\}^{\bb T_N}}\mu_N(\eta)\,
\log \frac{1}{[\alpha\wedge (1-\alpha)]^N}\\
&=& N\, (-\log[\alpha\wedge (1-\alpha)])\,.
\end{eqnarray*}
\end{proof}
\begin{proposition}\label{densidade}
Assume that $L$ is a reversible generator with respect to an invariant measure $\nu$ in a countable space-state
$E$, and $V:\bb R_+\times E\to \bb R$ is a bounded function. Notice that $L+V_t$ will be a symmetric operator
in $L^2(\nu)$. Denote by $\Gamma_t$ the largest eigenvalue of $L+V_t$:
\begin{equation*}
\Gamma_t=\sup_{\<f,f\>_{\nu}=1}\Big\{\<V_t,f^2\>_\nu+\<Lf,f\>_\nu\Big\}\,.
\end{equation*}
Then, the supremum above can be taken over only positive functions $f$, or else,
\begin{equation*}
\Gamma_t=\sup_{f\textrm{ density}}\Big\{\<V_t,(\sqrt{f})^2\>_\nu+\<L\sqrt{f},\sqrt{f}\>_\nu\Big\}\,.
\end{equation*}
\end{proposition}
\begin{proof}
It follows from the expression of the Dirichlet form (see \cite{kl}),
\begin{equation*}
\<Lf,f\>_\nu=-\pfrac{1}{2}\sum_{x,y\in E}\nu(x)\,L(x,y)[f(y)-f(x)]^2\,,
\end{equation*}
and the inequality $|\vert f(y)\vert-\vert f(x)\vert|\leq \vert f(y)-f(x)\vert$.
\end{proof}
\begin{proposition}\label{A3}
If $G_1$, $G_2$, $G_3$ are continuous functions defined in the torus $\bb T$, the application
from $D([0,T],\mc M)$ to $\bb R$ that associates to a trajectory $\{\pi_t: 0\leq t\leq T\}$ the number
\begin{eqnarray*}
 \sup_{0\le t\le T}
\Big\vert\, \<\pi_t, G_1\> \;-\; \<\pi_0, G_2 \> \;-\;
\int_0^t  \,  \<\pi_s ,G_3 \> \,ds \,\Big\vert
\end{eqnarray*}
is continuous for the Skorohod metric in $D([0,T],\mc M)$.
\end{proposition}

\begin{proof}
If $G$ is a continuous function in the torus, the application $\pi\mapsto \<\pi,G\>$ is a
continuous application from $\mc M$ to $\bb R$ in the weak topology.
From this observation and the definition of the Skorohod metric as an infimum under reparametrizations  (c.f. \cite{kl}),
the statement follows.
\end{proof}

\section*{Acknowledgements}
T.F. and A.N. thank Juan Gonzalez for pointing out the reference \cite{la}. 

P.G. thanks ``Funda\c c\~ao para a Ci\^encia e Tecnologia" for the research project with reference PTDC/MAT/109844/2009: ``Non-Equilibrium Statistical Physics" and for the financial support provided by the Research Center of
Mathematics of the University of Minho through the FCT Pluriannual
Funding Program. P.G. thanks the hospitality of ``Instituto de Matem\' atica Pura e Aplicada" where this work was initiated.

The authors thank Claudio Landim for nice discussions on the subject.

\end{document}